\documentclass[a4paper]{amsart}
\usepackage[utf8]{inputenc}
\usepackage[margin = 1in]{geometry}
\usepackage{amsmath, amssymb, amsfonts, amsthm, tensor, dsfont}
\usepackage[foot]{amsaddr}
\usepackage[svgnames]{xcolor}
\usepackage[style = alphabetic, sorting = nyt]{biblatex}

\usepackage{tikz-cd}
\usepackage[all]{xy}
\usepackage{hyperref}
\usepackage[hyphenbreaks]{breakurl}
\usepackage{subfiles}
\usepackage{verbatim}
\usepackage{comment}

\hypersetup{colorlinks = true, citecolor = MediumBlue, linkcolor = Crimson}

\newtheorem{thm}{Theorem}[section]
\newtheorem{lem}[thm]{Lemma}

\newtheorem{cor}[thm]{Corollary}

\theoremstyle{definition}

\newtheorem{mydef}[thm]{Definition}

\newcommand{\C}{\mathbb{C}}
\newcommand{\R}{\mathbb{R}}

\newcommand{\N}{\mathbb{N}}

\newcommand{\E}{\mathcal{E}}

\newcommand{\cH}{\mathcal{H}}
\newcommand{\cR}{\mathcal{R}}

\newcommand{\al}{\alpha}
\newcommand{\bb}{\beta}

\newcommand{\om}{\omega}
\newcommand{\ee}{\varepsilon}
\newcommand{\te}{\theta}
\newcommand{\tte}{\tilde{\theta}}
\newcommand{\la}{\lambda}

\newcommand{\dd}{\partial}
\newcommand{\dbar}{\bar\partial}

\newcommand{\tX}{\tilde{X}}

\newcommand{\PSH}{\text{PSH}}

\newcommand{\vol}{\text{Vol}}
\newcommand{\Poe}{P_{\om_{\ee}}}

\newcommand{\norm}[1]{\left\lVert#1\right\rVert}

\begin{document}

\title{Finite energy geodesic rays in big cohomology classes}
\author{Prakhar Gupta }
\address{Department of Mathematics, University of Maryland, College Park, Maryland, USA}
\email{pgupta8@umd.edu}

\begin{abstract}
For a big class represented by $\te$, we show that the metric space $(\E^{p}(X,\te),d_{p})$ for $p \geq 1$ is Buseman convex. This allows us to construct a chordal metric $d_{p}^{c}$ on the space of geodesic rays in $\E^{p}(X,\te)$. We also prove that the space of finite $p$-energy geodesic rays with the chordal metric $d_{p}^{c}$ is a complete geodesic metric space.

With the help of the metric $d_{p}$, we find a characterization of geodesic rays lying in $\E^{p}(X,\te)$ in terms of the corresponding test curves via the Ross-Witt Nystr\"om correspondence. This result is new even in the K\"ahler setting.

\end{abstract}



\maketitle

\tableofcontents

\section{Introduction} \label{sec: introduction}

%What to write in the introduction? It should be short, crisp, and to the point. 

Finding canonical metrics on a compact K\"ahler manifold $(X,\om)$ is a problem that has guided the field of K\"ahler geometry for many decades now. To tackle this problem, \cite{mabuchiriemannianstructure}, \cite{semmesriemannianmetric}, and \cite{donaldsonriemannianmetric} constructed a Riemannian structure on 
\[
\cH_{\om} = \{ u \in C^{\infty}(X) : \om + dd^{c}u > 0\},
\]
the space of smooth K\"ahler potentials in $\om$. To find the canonical metric cohomologous to $\om$, we needed to understand the geometry of $\cH_{\om}$ endowed with this Riemannian structure. In \cite{chenspaceofkahlermetrics}, Chen proved that this Riemannian structure gives rise to a metric $d_{2}$ on $\cH_{\om}$. In \cite{darvascompletionofspaceofkahlerpotentials}, Darvas showed that the completion $\overline{(\cH_{\om}, d_{2})}$ can be identified with $(\E^{2}(X,\om), d_{2})$, the space of finite energy potentials, confirming a conjecture of Guedj \cite{guedj2014metriccompletion}. Moreover, any two points in $\E^{2}(X,\om)$ can be joined by a metric geodesic lying in $\E^{2}(X,\om)$.  Using Finsler metric structure on $\cH_{\om}$ for $p \geq 1$, in \cite{darvasgeoemtryoffiniteenergy}, Darvas constructed complete geodesic metrics $(\E^{p}(X,\om), d_{p})$. These metrics proved useful in finding K\"ahler-Einstein \cite{darvasrubinsteinpropernessconjecture}, and cscK metrics in $(X,\om)$ \cite{chenchengcsckI}, \cite{chenchengcscKII}, \cite{chenandchengcscKIII}, \cite{BDL17}.

Further properties of the metric structure of $(\E^{p}(X,\om),d_{p})$ were studied to improve our understanding of canonical metrics. In \cite[Theorem 1.5]{chenandchengcscKIII}, Chen-Cheng proved that the metric space $(\E^{p}(X,\om), d_{p})$ is Buseman convex, with the case $p=1$ being proved in \cite{BDL17}. They use the Buseman convexity property for $p = 1$, to prove that the $L^1$ geodesic stability of $(X,\om)$ is equivalent to the existence of cscK metric cohomologous to $\om$ \cite{chenandchengcscKIII}. In \cite{DarvasLuuniformconvexity}, Darvas-Lu proved a uniform convexity property for $(\E^{p}(X,\om), d_{p})$ for $p > 1$, that allowed them to prove $C^{1,\bar{1}}$-geodesic stability for the existence of cscK metric. In op. cit., the authors used Buseman convexity and uniform convexity to prove that the space of $p$-finite energy geodesic rays $\cR^{p}_{\om}$ can be endowed with a complete geodesic metric $d_{p}^{c}$. 

In the author's previous paper \cite{guptacompletegeodescimetricinbigclasses}, he showed that for a big cohomology class $\{\te\}$, the space of finite energy potentials $\E^{p}(X,\te)$ can be endowed with a complete geodesic metric $d_{p}$. Moreover, he showed that the metric space $(\E^{p}(X,\te), d_{p})$ is uniformly convex if $p > 1$. With the prospects of studying stability in the big case in mind (c.f. \cite{darvas2023nonarchimedeanmetrics}), in this paper, we will explore the space of geodesic rays in the metric space $(\E^{p}(X,\te), d_{p})$. 

In \cite{guptacompletegeodescimetricinbigclasses}, the author used potentials of analytic singularity type to approximate the minimal singularity type. He used this approximation to construct the complete geodesic metric space $(\E^{p}(X,\te), d_{p})$ for $p \geq 1$. In this paper, we show that the same approximation scheme can be used to prove various properties of the metric space $(\E^{p}(X,\te),d_{p})$. 

As a first application of the approximation scheme, we prove a Lidskii-type inequality, the analog of a well known inequality for matrices (for a survey, see \cite[Theorem 2.7]{DarvasLuRubinsteinQuantization}). This result was first proved in \cite[Theorem 5.1]{DarvasLuRubinsteinQuantization} for K\"ahler classes $\om = c_{1}(L)$ induced by an ample line bundle $L$. In \cite[Corollary 3.2]{DarvasLuuniformconvexity}, Darvas-Lu extended this result to an arbitrary K\"ahler class. In this paper, we use the approximation scheme to extend the result to arbitrary big classes. In particular, we prove
\begin{thm}\label{thm: introduction lidskii type inequality}
    If $u,v,w \in \E^{p}(X,\te)$  satisfy $u \geq v \geq w$, then
    \[
    d_{p}^{p}(v,w) \leq d_{p}^{p}(u,w) - d_{p}^{p}(u,v).
    \]
\end{thm}

Next, we extend \cite[Theorem 1.5]{chenandchengcscKIII} of Chen-Cheng to the big cohomology classes. In the K\"ahler setting, they prove a Buseman type convexity of the metric space $(\E^{p}(X, \om), d_{p})$. We prove

\begin{thm}\label{thm: introduction buseman convexity}
    Let $\te$ be a real, smooth, closed $(1,1)$-form representing a big cohomology class. If $p \geq 1$, then the metric space $(\E^{p}(X,\te), d_{p})$ is Buseman convex. This means if $u_{0}, u_{1}, v_{0}, v_{1} \in \E^{p}(X,\te)$, and $[0,1] \ni t \mapsto u_{t}$ and $[0,1]\ni t \mapsto v_{t}$ are the weak geodesics joining $u_{0}$, $u_{1}$, and $v_{0}, v_{1}$ respectively, then 
    \[
    d_{p}(u_{t}, v_{t}) \leq (1-t) d_{p}(u_{0}, v_{0}) + t d_{p}(u_{1}, v_{1}). 
    \]
\end{thm}

These results allow us to study the metric geometry of the space of geodesic rays in $\E^{p}(X,\te)$ in more depth. Given $u \in \E^{p}(X,\te)$, we say the geodesic ray $[0,\infty) \ni t \mapsto u_{t} \in \E^{p}(X,\te) $  is a finite $p$-energy geodesic ray starting at $u_{0} = u$.  We denote such a ray by $\{u_{t}\}_{t} \in \cR^{p}_{u}$. We can endow $\cR^{p}_{u}$ with a metric, as follows. Given geodesic rays $\{u_{t}^{0} \}_{t}, \{u_{t}^{1}\}_{t} \in \cR^{p}_{u}$ we define the chordal distance between them by 
\begin{equation}\label{eq: chordal metric for geodesic rays}
d_{p}^{c}(\{u_{t}^{0}\}_{t}, \{u_{t}^{1}\}_{t}) = \lim_{t \to \infty} \frac{d_{p}(u_{t}^{0}, u_{t}^{1})}{t}.    
\end{equation}

Let $\te$ represent a big cohomology class. By $\cR^{p}_{u}$ we denote the space of finite $p$-energy geodesic rays emanating from $u \in \E^{p}(X,\te)$.

\begin{thm}\label{thm: Introduction finite energy geodesic ray}
 $(\cR^{p}_{u}, d_{p}^{c})$ is a complete geodesic metric space. 
\end{thm}

%In the K\"ahler case, this metric was introduced by Darvas-Lu in \cite{DarvasLuuniformconvexity} to study the $C^{1,\bar{1}}$ geodesic stability of the Mabuchi functional. We will follow the proof in \cite[Section 4]{DarvasLuuniformconvexity} to extend their result in the big case to prove Theorem~\ref{thm: Introduction finite energy geodesic ray}.

\subsection{Ross-Witt Nystr\"om correspondence} \label{subsec: ross witt nystrom in introduction}
 The notion of test curves was introduced in \cite{RossWittNystromanalytictestconfigurations} by Ross-Witt Nystr\"om in the K\"ahler setting, giving a potential theoretic framework to study stability.

 Test curves are dual to (sub)geodesic rays via the Legendre transform. Ross-Witt Nystr\"oms's work generalized the work of Phong-Sturm \cite{phongsturmtestconfigurations} where they associated geodesic rays to test configurations. The study of test curves was extended to the big setting in \cite{darvas2021l1} by Darvas-Di Nezza-Lu. Furthermore, the notion of finite energy test curves was introduced in \cite{darvasxiaclosureoftestconfigurations} (in the K\"ahler setting) and \cite{darvas2023twisted} (in the big setting). Test curves can detect several properties of the geodesic rays. Using the metric $d_{p}$ on $\E^{p}(X,\te)$, the following theorem proves that the geodesic rays lying in $\E^{p}(X,\te)$ can be detected from the corresponding test curve.

If $(0,\infty) \ni t \mapsto u_{t} \in \PSH(X,\te)$  is a $\te$-geodesic ray starting at $V_{\te}$, then its Legendre dual $ \R \ni \tau \mapsto \hat{u}_{\tau} \in \PSH(X,\te)$ is defined by 
\[
\hat{u}_{\tau} = \inf_{t > 0} (u_{t} - t\tau).
\]

For $\tau$ large enough, $\hat{u}_{\tau} \equiv -\infty$. Thus we define, 
\[
\tau_{\hat{u}}^{+} := \inf \{ \tau \in \R \, : \, \hat{u}_{\tau} \equiv -\infty\} < \infty. 
\]

 \begin{thm}\label{thm: ross witt nystrom in introduction}
     Let $\{u_{t}\}_{t}$ be a $\te$-geodesic ray starting from $V_{\te}$. Let $\{\hat{u}_{\tau}\}_{\tau}$ be the Legendre dual of $\{u_{t}\}_{t}$. Then $u_{t} \in \E^{p}(X, \te)$ for all $t \geq 0$ iff 
    \begin{equation}\label{eq: p energy of test curve in intro}
             \int_{-\infty}^{\tau^{+}_{\hat{u}}}(-\tau + \tau^{+}_{\hat{u}})^{p-1}\left(\int_{X}\te^{n}_{V_{\te}} - \int_{X}\te^{n}_{\hat{u}_{\tau}} \right) d\tau < \infty
    \end{equation}

 \end{thm}
When $\tau^{+}_{\hat{u}} = 0$, the expression in Equation~\eqref{eq: p energy of test curve in intro} equals $\frac{1}{p}d_{p}^{p} (V_{\te}, u_{1}) = \frac{1}{p}\int_{X}|\dot{u}_{0}|^{p}\te^{n}_{V_{\te}}$. Thus this expression is closely related to the speed of the geodesic ray in $\E^{p}(X,\te)$. When $p = 1$, such characterization of finite energy geodesic rays was obtained in \cite[Theorem 3.7]{darvasxiaclosureoftestconfigurations} (in the K\"ahler setting) and \cite[Theorem 3.9]{darvas2023twisted} (in the big setting). For $p > 1$, this is new in the K\"ahler case as well. 

\subsection*{Organization} In Section~\ref{sec: preliminaries}, we recall some background material and in particular the approximation scheme that is used in the construction of the metric $d_{p}$ on $\E^{p}(X,\te)$. In Section~\ref{sec: Lidskii type inequality}, we prove Theorem~\ref{thm: introduction lidskii type inequality}.  In Section~\ref{sec: convexity}, we prove Theorem~\ref{thm: introduction buseman convexity}. In Section~\ref{sec: space of geodesic rays}, we prove Theorem~\ref{thm: Introduction finite energy geodesic ray}. Lastly, in Section~\ref{sec: Ross-Witt Nystrom correspondence}, we prove theorem~\ref{thm: ross witt nystrom in introduction}.

\subsection*{Acknowledgement}I want to thank my advisor Tam\'as Darvas for his constant support and valuable suggestions on an early draft that improved the paper.  I am also grateful to Mingchen Xia and Antonio Trusiani for carefully reading the preliminary version of this paper and suggesting helpful improvements. 

Research is partially supported by NSF CAREER grant DMS-1846942.

\section{Preliminaries} \label{sec: preliminaries}

\subsection{Pluripotential theory}\label{sec: facts of pluripotential theory}
In this section, we will recall the notions from the pluripotential theory developed in \cite{Boucksom2008MongeAmpreEI}. Throughout, we will work on a compact K\"ahler manifold $(X,\om)$. 

Let $\te$ be a smooth closed real $(1,1)$-from representing the cohomology class $\{\te\} \in H^{1,1}(X,\R)$. A function $u : X \to \R\cup \{-\infty\}$ is called a $\te$-psh function if locally $u + g$ is plurisubharmonic, where $\te = dd^{c}g$. By $dd^{c}$ here we mean $\sqrt{-1}\dd \dbar/2\pi$. We use $\PSH(X,\te)$ to denote the set of all $\te$-psh functions. If $u$ is $\te$-psh, then the $(1,1)$-current $\te_{u} := \te + dd^{c}u$ is a closed positive current. 

We say $\te$ represents a pseudoeffective class if $\PSH(X,\te) \neq \emptyset$. We say $\te$ represents a big cohomology class if there exists $u \in \PSH(X,\te)$ such that $\te_{u} \geq \ee \om$ for some small enough $\ee > 0$. We say that $\te$ represents a K\"ahler class if there is $u \in \PSH(X,\te) \cap C^{\infty}(X)$ such that $\te_{u} \geq \ee \om$ for some $\ee > 0$. We say that $\te$ represents a nef cohomology class $\{\te + \ee \om\}$ is a K\"ahler class.  In a K\"ahler class, there are plenty of smooth potentials, but in an arbitrary big class, we do not expect any smooth potentials.

Let $\te$ represent a big cohomology class. If $u, v \in \PSH(X,\te)$ we say that $u$ is more singular than $v$, denoted by $u \preceq v$ if $u \leq v + C$ for some constant $C$. We say that $u, v \in \PSH(X,\te)$ have the same singularity type if $u \preceq v$ and $v \preceq u$. The potential
\begin{equation}\label{eq: def of V_theta}
V_{\te} = \sup\{ u \in \PSH(X,\te) : u \leq 0\}    
\end{equation}

has the least singularity among all $\te$-psh functions. Any potential with the same singularity type as $V_{\te}$ is called a minimal singularity potential. 

In \cite{Boucksom2008MongeAmpreEI}, the authors introduced a non-pluripolar measure $\te^{n}_{u}:= (\te + dd^{c}u)^{n}$ associated to any $\te$-psh function $u$. Witt Nystr\"om proved \cite{Wittnystrommonotonicity} proved that if $u \preceq v$, then $\int_{X} \te^{n}_{u} \leq \int_{X} \te^{n}_{v}$. Thus for any $u \in \PSH(X,\te)$, $\int_{X} \te^{n}_{u} \leq \int_{X}\te^{n}_{V_{\te}} := \vol(\te)$. 

We denote the potentials of full mass by $\E(X,\te)$. In particular
\[
\E(X,\te) := \left\{ u \in \PSH(X,\te) \, : \, \int_{X} \te^{n}_{u} = \int_{X}\te^{n}_{V_{\te}}\right\}.
\]

For $p \geq 1$, we say that $u \in \E(X,\te)$ has finite $p$-energy if $\int_{X}|u-V_{\te}|^{p}\te^{n}_{u} < \infty$. We denote
\[
\E^{p}(X,\te) := \left\{ u \in \E(X,\te) \,: \,  \int_{X} |u-V_{\te}|^{p}\te^{n}_{u} < \infty\right\}.
\]

We can do the same construction in the prescribed singularity setting as well. See \cite{darvas2023relative} for more details. If $\phi$ is a model potential, meaning $P_{\te}[\phi](0) = \phi$ (see Equation~\eqref{eq: envelope with arbitrary singularity type}), then we define the space of relative full mass as
\[
\E(X,\te, \phi) = \left\{ u \in \PSH(X,\te) \, : \, u \preceq \phi, \int_{X}\te^{n}_{u} = \int_{X} \te^{n}_{\phi}\right\}
\]
and the space of relatively finite $p$-energy as 
\[
\E^{p}(X,\te ,\phi) = \left\{ u \in \E(X,\te,\phi) : \int_{X} |u-\phi|^{p} \te^{n}_{\phi} < \infty \right\}.
\]

\subsection{Geodesic rays}\label{sec: geodesic rays preliminaries}
Following Berndtsson \cite{Berndtssonweakgeodesics} and Darvas-Di Nezza-Lu \cite{Darvas2017OnTS} we define the geodesic segments by an envelope construction. Let $S \subset \C$ be a vertical strip given by $S = \{ z \in \C \, | \, 0 \leq \text{Re}(z) \leq 1\}$. Let $\pi : S \times X \to X$ be the projection. Let $u_{0}, u_{1} \in \PSH(X,\te)$, then we say that a path $(0,1) \ni t \mapsto \PSH(X,\te)$ is a subgeodesic between $u_{0}$ and $u_{1}$ if the map $S \times X \ni (z,x) \mapsto V(z,x) := v_{\text{Re}(z)}(x)$ is $\pi^{*}\te$-psh and $\lim_{t \to 0,1} v_{t} \leq u_{0,1}$. Let $\mathcal{S}(u_{0},u_{1})$ to be the collection of all subgeodesics between $u_{0}$ and $u_{1}$. We define the geodesic joining $u_{0}$ and $u_{1}$ to be a path $(0,1) \ni t \mapsto u_{t} \in \PSH(X,\te)$ given by 
\begin{equation}\label{eq: def of geodesic}
u_{t}(x) = \sup_{\mathcal{S}(u_{0}, u_{1})} v_{t}(x)    
\end{equation}

In the K\"ahler setting, the geodesics joining ``smooth'' points have some regularity properties. If $\eta$ represents a K\"ahler class (and not necessarily be a K\"ahler form), then we denote 
\[
\cH_{\eta}^{1,\bar{1}} = \PSH(X,\eta) \cap C^{1,\bar{1}}(X)
\]
which will act as the space of ``smooth'' potentials for us.

We recall that in \cite{HespaceofKahlerpotentials}, He proved 
\begin{lem}\label{lem: C11bar regularity for Kahler case}
    If $\om$ is a K\"ahler form and $u_{0}, u_{1} \in \cH_{\om}^{1,\bar{1}}$, then the geodesic $u_{t}$ joining $u_{0}$ and $u_{1}$ is also in $\cH^{1,\bar{1}}_{\om}$. 
\end{lem}
We can modify He's result to prove
\begin{lem}\label{lem: C11bar regularity for eta}
    If $\eta$ represents a K\"ahler class and $u_{0},u_{1} \in \cH^{1,\bar{1}}_{\eta}$, then the geodesic $u_{t}$ joining them is in $\cH^{1,\bar{1}}_{\eta}$ as well.
\end{lem}
\begin{proof}
    Let $g \in C^{\infty}(X)$ such that $\eta_{g}:= \eta + dd^{c}g$ is K\"ahler. Then $u_{0} - g, u_{1}- g \in C^{1,\bar{1}}(X)$ are $\eta_{g}$-psh. Lemma~\ref{lem: C11bar regularity for Kahler case} says that the goedesic $u_{t,g}$ joining them is in $C^{1,\bar{1}}(X)$. Since $u_{t,g} + g$ is the $\eta$-geodesic joining $u_{0}$ and $u_{1}$, we get that $u_{t} = u_{t,g} + g \in C^{1,\bar{1}}(X)$ as well.
\end{proof}

Now we define geodesic rays. Again we assume that $\te$ represents a big cohomology class. A path $(0,\infty) \ni t \mapsto v_{t} \in \PSH(X,\te)$ is a \emph{sublinear subgeodesic ray} starting at $u_{0} \in \PSH(X,\te)$ if $u_{t} \to u_{0}$ in $L^{1}(X)$ as $ t \to 0$ and for any $0 < a < b < \infty$, the path $(a,b) \ni t \mapsto u_{t}$ is a subgoedesic between $u_{a}$ and $u_{b}$ and $u_{t} \leq u_{0} + Ct$ for some constant $C$. 

A sublinear subgeodesic ray $(0,\infty) \ni t \mapsto u_{t} \in \PSH(X,\te)$ is a \emph{geodesic ray} starting at $u_{0} \in \PSH(X,\te)$ if $u_{t} \to u_{0}$ in $L^{1}(X)$ as $t \to 0$ and for any $0 < a < b < \infty$, the path $(a, b) \ni t \mapsto u_{t} $ is a geodesic ray joining $u_{a}$ and $u_{b}$. We recall a few useful results about the geodesic rays.

\begin{lem}[{\cite[Remark 3.3]{darvas2023twisted}}]\label{lem: sup is affine in weak geodesics starting from V_{te}}
    If $[0,\infty) \ni t \mapsto u_{t} \in \E^{p}(X,\te)$ is a geodesic ray starting from $u_{0} = V_{\te}$, then $t \mapsto \sup_{X}u_{t} = \sup_{\text{Amp}(\te)}(u_{t} - V_{\te})$ is linear. 
\end{lem}

\begin{thm}[{\cite[Proposition 3.8]{darvas2023twisted}}]\label{thm: approximting with minimal singularity type}
    If $[0, \infty) \ni t \mapsto u_{t} \in \E^{p}(X,\te)$ is a geodesic ray starting at $u_{0} = V_{\te}$, then one can find geodesic rays of minimal singularity type $[0,\infty) \ni t \mapsto u_{t}^{j} \in \E^{p}(X,\te)$ such that each $u_{t}^{j}$ has minimal singularity and $u_{t}^{j} \searrow u_{t}$ as $j \to \infty$ for all $t \geq 0$. 
\end{thm}

\subsection{Plurisubharmonic envelopes}\label{sec: plurisubharmonic envelopes}

The envelope construction has several applications in pluripotential theory. We have already seen two examples of such construction in Equations~\ref{eq: def of V_theta} and~\ref{eq: def of geodesic}. In this section we will see more such examples are recall some theorems about them. 

Assume that $\te$ is a smooth closed real $(1,1)$-from that represents a big cohomology class. Given a measurable function $f :  X \to \R\cup \{\pm \infty\}$, we define
\begin{equation}\label{eq: envelope with minimal singularity type}
P_{\te}(f) := (\sup\{ u \in \PSH(X,\te) \,: \, u \leq f\})^{*},
\end{equation}
where $\varphi^{*}$ is the upper semicontinuous regularization of $\varphi$. $P_{\te}(f) \equiv - \infty$ if there is no $\te$-psh function such that $u \leq f$. 

If $\phi \in \PSH(X,\te)$, then the envelope with respect to the singularity type of $\phi$ is constructed by 
\begin{equation}\label{eq: envelope with arbitrary singularity type}
P_{\te}[\phi](f) := (\sup\{ u \in \PSH(X,\te) \, : \, u \preceq \phi, u \leq f\})^{*} = \left(\lim_{C \to \infty} P_{\te}(\phi + C, f)\right)^{*}.    
\end{equation}

In the K\"ahler setting, we have the following regularity result.

\begin{lem}\label{lem: envelope for eta}
    If $\eta$ represents a K\"ahler class (and not necessarily be a K\"ahler form) and $f \in C^{1,\bar{1}}(X)$, then $P_{\eta}(f) \in C^{1,\bar{1}}(X)$. 
\end{lem}
\begin{proof}
    Let $g \in C^{\infty}(X)$ such that $\eta_{g} = \eta + dd^{c}g$ is a K\"ahler form. We can see that $P_{\eta}(f) = P_{\eta_{g}}(f-g) + g$.  Now due to \cite{BermanMAequationzerotemplimit} (for a survey see \cite[Appendix A.3]{Darvas2019GeometricPT}), $P_{\eta_{g}}(f-g) \in C^{1,\bar{1}}(X)$ as $f-g \in C^{1,\bar{1}}(X)$. Therefore, $P_{\eta_{g}}(f-g) + g \in C^{1,\bar{1}}(X)$. Thus $P_{\eta}(f) \in C^{1,\bar{1}}(X)$. 
\end{proof}
In general, we have the following formula for the non-pluripolar measures of potentials obtained by taking an envelope :
\begin{lem}[{\cite{EleonoraTrapaniMAmeasureoncontactsets}}]\label{lem: measure on contact set} If $\te$ represents a big cohomology class, $\phi \in \PSH(X,\te)$, and $f \in C^{1,\bar{1}}(X)$, then 
\[
\te^{n}_{P_{\te}[\phi](f)} = \mathds{1}_{\{ P_{\te}[\phi](f) = f\}} \te^{n}_{f}.
\]
    
\end{lem}

\subsection{The complete geodesic metric $d_{p}$.}
If $\om$ is a K\"ahler metric, a complete geodesic metric $d_{p}$ on $\E^{p}(X,\om)$ was constructed by Darvas in \cite{darvasgeoemtryoffiniteenergy}. If $\bb$ represents a big and nef cohomology class, then by approximating from the K\"ahler case, Di Nezza-Lu constructed such a complete geodesic metric $d_{p}$ on $\E^{p}(X,\bb)$ in \cite{dinezza2018lp}. If $\te$ represents a big cohomology class, then in \cite{guptacompletegeodescimetricinbigclasses}, the author constructed an approximation scheme via analytic singularity types to construct a complete geodesic metric $d_{p}$ on $\E^{p}(X,\te)$. In this subsection, we will describe the approximation method used to construct $d_{p}$ on big and nef, and big classes.

\subsubsection{$d_{p}$ metric in big and nef classes} \label{sec: dp metric in big and nef}
Let $\bb$ represent a big and nef cohomology class. Here we explain the construction of the complete geodesic metric $d_{p}$ on $\E^{p}(X,\bb)$ by Di Nezza-Lu in \cite{dinezza2018lp}. Let $\om$ be a K\"ahler form.  Denote by $\om_{\ee}:= \bb + \ee \om$, a smooth form which represents a K\"ahler class. There is a complete geodesic metric $d_{p}$ on $\E^{p}(X,\om_{\ee})$. Let
\[
\cH_{\bb} = \{ u \in \PSH(X,\bb)  :  u = P_{\bb}(f) \text{ where } f \in C^{1,\bar{1}}(X)\}.
\]
$\cH_{\bb}$ is the analog of the space of smooth potentials from the K\"ahler case, and it was first used in \cite{dinezza2018lp}. If $u_{0}, u_{1} \in \cH_{\bb}$, then there are $f_{0}, f_{1} \in C^{1,\bar{1}}(X)$ such that $u_{0} = P_{\bb}(f_{0})$ and $u_{1} = P_{\bb}(f_{1})$. We define $u_{0,\ee} = \Poe(f_{0})$ and $u_{1,\ee} = \Poe(f_{1})$. We define 
\[
d_{p}(u_{0},u_{1}) = \lim_{\ee \to 0} d_{p}(u_{0,\ee}, u_{1,\ee}).
\]
For $u_{0}, u_{1} \in \E^{p}(X,\bb)$, we can find $u_{0}^{j}, u_{1}^{j} \in \cH_{\bb}$ such that $u_{0}^{j} \searrow u_{0}$ and $u_{1}^{j} \searrow u_{1}$. We define 
\[
d_{p}(u_{0}, u_{1}) = \lim_{j\to \infty}d_{p}(u_{0}^{j}, u_{1}^{j}).
\]

\subsubsection{$d_{p}$ metric in the analytic singularity setting} \label{sec: dp metric in analytic singularity}
Let $\psi$ be a $\te$-psh function with analytic singularities. By Hironaka's resolution, we can find a modification $\mu: \tX \to X$ that resolves the singularities of $\psi$. See \cite[Section 3]{guptacompletegeodescimetricinbigclasses} for more details. There is a smooth closed real $(1,1)$-form $\tte$ and a bounded $\tte$-psh function $g$ on $\tX$ such that the map 
\[
\PSH(X,\te,\psi) \ni u \mapsto \tilde{u} := (u-\psi)\circ \mu + g \in \PSH(\tX, \tte)
\]
is an order-preserving bijection. Moreover, this mapping is a bijection between $\E^{p}(X,\te,\psi)$ and $\E^{p}(\tX, \tte)$ as well. Since $\tte$ represents a big and nef class, there is a complete geodesic metric $d_{p}$ on $\E^{p}(\tX, \tte)$. Using this correspondence, we can define the metric on $\E^{p}(X,\te,\psi)$ as follows. Let $u_{0}, u_{1} \in \E^{p}(X,\te, \psi)$, then 
\begin{equation}\label{eq: equation for dp in the new paper}
    d_{p}(u_{0}, u_{1}) := d_{p}(\tilde{u}_{0}, \tilde{u}_{1}). 
\end{equation}

\subsubsection{$d_{p}$ metric in big classes}\label{sec: dp metric in big}
Let $\te$ represent a big cohomology class. By Demailly's regularization, we can find an increasing sequence $\psi_{k}$ of $\te$-psh functions with analytic singularities such that $\psi_{k} \nearrow V_{\te}$.

By \cite[Theorem 7.4]{guptacompletegeodescimetricinbigclasses}, we can define the complete geodesic metric $d_{p}$ on $\E^{p}(X,\te)$ by the following. If $u_{0}, u_{1} \in \E^{p}(X,\te)$, then
\begin{equation} \label{eq: definition of dp for big class}
d_{p}(u_{0}, u_{1}) = \lim_{k \to \infty} d_{p}(P_{\te}[\psi_{k}](u_{0}), P_{\te}[\psi_{k}](u_{1})).    
\end{equation}

\subsubsection{Properties of the metric $d_{p}$}

Let $\te$ represent a big cohomology class. We denote by 
\[
\cH_{\te} = \{ u \in \PSH(X,\te) : u = P_{\te}(f) \text{ for some } f \in C^{1,\bar{1}}(X)\}.
\]
Notice that if $\te$ represents a K\"ahler class, then due to Lemma~\ref{lem: envelope for eta} $\cH_{\te}$ consists of $C^{1,\bar{1}}$ potentials. 

\begin{thm}[ {\cite[Theorem 4.7]{guptacompletegeodescimetricinbigclasses}}] \label{thm: geodesic speed and distance}
    Let $u_{0}, u_{1} \in \cH_{\te}$, and let $u_{t}$ be the weak geodesic joining $u_{0}$ and $u_{1}$. Then 
    \[
    d_{p}^{p}(u_{0},u_{1}) = \int_{X}|\dot{u}_{0}|^{p}\te^{n}_{u_{0}} = \int_{X}|\dot{u}_{1}|^{p}\te^{n}_{u_{1}}.
    \]
\end{thm}

\begin{lem}[Pythagorean identity, {\cite[Theorem 5.5]{guptacompletegeodescimetricinbigclasses}}] \label{lem: Pythagorean identity}
\sloppy
The metric $d_{p}$ satisfies the following Pythagorean identity. If $u_{0}, u_{1} \in \E^{p}(X,\te)$, then
\[
d_{p}^{p}(u_{0}, u_{1}) = d_{p}^{p}(u_{0}, P_{\te}(u_{0},u_{1})) + d_{p}^{p}(u_{1}, P_{\te}(u_{0},u_{1})).
\]
\end{lem}

\begin{lem}[{\cite[Lemma 5.2]{guptacompletegeodescimetricinbigclasses}}]\label{lem: continuity under decreasing sequences}
    If $u_{0}^{j}, u_{1}^{j} , u_{0}, u_{1} \in \E^{p}(X,\te)$ satisfy $u_{0}^{j} \searrow u_{0}$ and $u_{1}^{j} \searrow u_{1}$, then 
    \[
    \lim_{j \to \infty}d_{p}(u_{0}^{j}, u_{1}^{j}) = d_{p}(u_{0}, u_{1}).
    \]
\end{lem}

\begin{thm}[Uniform Convexity, {\cite[Theorem 8.3]{guptacompletegeodescimetricinbigclasses}}] \label{thm: uniform convexity}
    Let $p > 1$. Let $u, v_{0}, v_{1} \in \E^{p}(X,\te)$ and $v_{\la}$ be the weak geodesic joining $v_{0}$ and $v_{1}$. Then
    \begin{align*}
        d_{p}(u,v_{\lambda})^{2} &\leq (1-\la)d_{p}(u,v_{0})^{2} + \la d_{p}(u,v_{1})^{2} - (p-1)\la(1-\la)d_{p}(v_{0},v_{1})^{2}, \text{ if } 1< p \leq 2 \text{ and }\\
d_{p}(u,v_{\la})^{p} &\leq (1-\la)d_{p}(u,v_{0})^{p} + \la d_{p}(u,v_{1})^{p} - \la^{p/2}(1-\la)^{p/2}d_{p}(v_{0},v_{1})^{p}, \text{ if } p \leq 2.
    \end{align*} 
\end{thm}

A corollary of Theorem~\ref{thm: uniform convexity} is 
\begin{cor}[{\cite[Corollary 8.5]{guptacompletegeodescimetricinbigclasses}}]\label{cor: uniform convexity}
    For $p > 1$, if $u,v_{0},v_{1} \in \E^{p}(X,\te)$ such that for some $\la \in [0,1]$ and $\ee > 0$, $d_{p}(u, v_{0}) \leq (\la + \ee) d_{p}(v_{0}, v_{1})$ and $d_{p}(u,v_{1}) \leq (1-\la + \ee) d_{p}(v_{0},v_{1})$, then there is a constant $C > 0$ such that
    \[
    d_{p}(u,v_{\la}) \leq C\ee^{\frac{1}{r}}d_{p}(v_{0},v_{1})
    \]
    where $r = \max\{2,p\}$ and $v_{t}$ is the geodesic joining $v_{0}$ and $v_{1}$. 
\end{cor}

A consequence of these properties is the following 

\begin{lem} \label{lem: one side inequality for dp distances}
    If $1\leq p' \leq p < \infty$ then for any $u_{0}, u_{1} \in \E^{p}(X,\te)$, 
    \[
    d_{p'}(u_{0}, u_{1}) \leq d_{p}(u_{0}, u_{1}) \vol(\te)^{\frac{1}{p'} - \frac{1}{p}}.
    \]
\end{lem}
\begin{proof}
    First we assume that $u_{0}, u_{1} \in \cH_{\te}$ and $u_{t}$ is the weak geodesic joining $u_{0}$ and $u_{1}$. Then by Theorem~\ref{thm: geodesic speed and distance},
    \[
    d_{p}^{p}(u_{0}, u_{1}) = \int_{X} |\dot{u}_{0}|^{p}\te^{n}_{u_{0}} \qquad \text{ and } \qquad d_{p'}^{p'}(u_{0}, u_{1}) = \int_{X} |\dot{u}_{0}|^{p'}\te^{n}_{u_{0}}.
    \]
    Applying H\"older inequality with factors $p/p'$ and $p/(p-p')$ we can write,
    \[
    \int_{X} |\dot{u}_{0}|^{p'}\te^{n}_{u_{0}} \leq \left(\int_{X}|\dot{u}_{0}|^{p}\te^{n}_{u_{0}}\right)^{p'/p} \left( \int_{X} \te^{n}_{u_{0}}\right)^{\frac{p-p'}{p}}.
    \]
    Raising both sides to $1/p'$ and noticing that $\int_{X}\te^{n}_{u_{0}} = \vol(\te)$, we get that 
    \[
    d_{p'}(u_{0}, u_{1}) \leq d_{p}(u_{0},u_{1})\vol(\te)^{\frac{1}{p'} - \frac{1}{p}}.
    \]

    More generally, if $u_{0}, u_{1} \in \E^{p}(X,\te)$, then we can find $u_{0}^{j}, u_{1}^{j} \in \cH_{\te}$ such that $u_{0}^{j} \searrow u_{0}$ and $u_{1}^{j} \searrow u_{1}$. By Lemma~\ref{lem: continuity under decreasing sequences}, 
    \[
    \lim_{j \to \infty} d_{p'}(u_{0}^{j}, u_{1}^{j}) = d_{p'}(u_{0}, u_{1}) \qquad \text{ and } \qquad \lim_{j \to \infty} d_{p}(u_{0}^{j}, u_{1}^{j}) = d_{p}(u_{0}, u_{1}).
    \]
    From the proof above we can write,
    \[
    d_{p'}(u_{0}, u_{1}) = \lim_{j \to \infty} d_{p'}(u_{0}^{j}, u_{1}^{j}) \leq \lim_{j \to \infty} d_{p}(u_{0}^{j}, u_{1}^{j}) \vol(\te)^{\frac{1}{p'} - \frac{1}{p}} = d_{p}(u_{0}, u_{1}) \vol(\te)^{\frac{1}{p'}- \frac{1}{p}}.
    \]
\end{proof}

The following Lemma from \cite{BDL17} shows that the weak geodesics are stable under perturbations of the endpoints. 
\begin{lem}\label{lem: endpoint stability}
    If $u_{0}^{j}, u_{1}^{j}, u_{0}, u_{1} \in \E^{p}(X,\te)$ satisfy $d_{p}(u_{0}^{j}, u_{0}) \to 0$ and $d_{p}(u_{1}^{j}, u_{1}) \to 0$ as $j \to \infty$, then $d_{p}(u_{t}^{j}, u_{t}) \to 0$ as $j \to \infty$ where $u_{t}^{j}$, $u_{t}$ are the weak geodesics joining $u_{0}^{j}, u_{1}^{j}$ and $u_{0}, u_{1}$ respectively. 
\end{lem}

\begin{proof}
    The proof is the same as in \cite[Proposition 4.3]{BDL17}. We just need the corresponding results used in the proof in \cite[Proposition 4.3]{BDL17} for the big setting. For that we can obtain the quasi-monotonicity for the $d_{p}$ metric from \cite[Theorem 5.2]{GuptaCompletemetrictopology} and \cite[Lemma 4.12]{guptacompletegeodescimetricinbigclasses}. The fact that for $u,v \in \E^{1}(X,\te)$ satisfying $u\leq v$ and $I(u) = I(v)$ imply $u = v$ follows by combining \cite[Theorem 2.4, Proposition 2.5, and Proposition 2.8]{darvas2021l1}. Moreover, the fact that $u \leq v \leq w$ implies $d_{p}(u,v) \leq d_{p}(u,w)$ can be obtained from Theorem~\ref{thm: introduction lidskii type inequality} (whose proof does not rely on this lemma).
\end{proof}

\begin{lem}\label{lem: decreasing bounded dp sequence has a limit in Ep}
    If $u_{j} \in \E^{p}(X,\te)$ is a $d_{p}$-bounded decreasing sequence, then $\lim_{j \to \infty} u_{j}=: u \in \E^{p}(X,\te)$. 
\end{lem}

\begin{proof}
    By \cite[Lemma 4.12]{guptacompletegeodescimetricinbigclasses}, the metric $d_{p}$ is comparable to the quasi-metric $I_{p}$ i.e., there exists $C > 1$, such that 
    \[
    \frac{1}{C}I_{p}(u_{0}, u_{1}) \leq d_{p}^{p}(u_{0}, u_{1}) \leq CI_{p}(u_{0},u_{1}).
    \]
    Here the functional $I_{p}$ is given by 
    \[
    I_{p}(u_{0}, u_{1}) = \int_{X}|u_{0} - u_{1}|^{p} (\te^{n}_{u_{0}} + \te^{n}_{u_{1}})
    \]
    for $u_{0}, u_{1} \in \E^{p}(X,\te)$. So, if the sequence $u_{j}$ is $d_{p}$-bounded, then it is $I_{p}$-bounded as well. This means that 
    \[
    \sup_{j \in \N} \int_{X}|u_{0} - u_{j}|^{p}\te^{n}_{u_{j}} < \infty.
    \]
    Thus by \cite{Guedj2019PlurisubharmonicEA}, the limit $\lim_{j \to \infty} u_{j} =: u \in \E^{p}(X,\te)$.
\end{proof}

We will end this section by proving a $d_p$-contraction property for the projection into different cohomology classes. This result will play an important role in obtaining our main results. This is an analog of the contraction property in \cite[Theorem 7.3]{guptacompletegeodescimetricinbigclasses}, and \cite[Theorem 4.3]{TrusianiL1metric} for $p = 1$, where the projection operator maps to a different singularity type in the same cohomology class. 

\begin{lem}[Contraction property]\label{lem: a general contraction property}
    If $\te$ and $\eta$ are closed smooth real $(1,1)$-forms representing big cohomology classes such that $\theta \leq \eta$. Then the map $P_{\te} : \mathcal{H}_{\eta} \mapsto \mathcal{H}_{\te}$ is a contraction. This means that $v_{0}, v_{1} \in \cH_{\eta}$, then 
    \[
    d_{p}(P_{\te}(v_{0}), P_{\te}(v_{1})) \leq d_{p}(v_{0},v_{1}).
    \]
\end{lem}

\begin{proof}
    First we will show that $P_{\te}$ maps $\cH_{\eta}$ to $\cH_{\te}$. If $v \in \cH_{\eta}$, then $v = P_{\eta}(f)$ for some $f \in C^{1,\bar{1}}$. By a standard argument, we can show that $P_{\te}(v) = P_{\te}(f)$. Thus showing that $P_{\te}(v) \in \cH_{\te}$. 

   Now, assume that $f_{0}, f_{1} \in C^{1,\bar{1}}$ be such that $v_{0} = P_{\eta}(f_{0})$ and $v_{1} = P_{\eta}(f_{1})$. Also call $u_{0} = P_{\te}(f_{0}) = P_{\te}(v_{0})$ and $u_{1} = P_{\te}(f_{1}) = P_{\te}(v_{1})$. For now, we also assume that $f_{0} \leq f_{1}$. By Theorem~\ref{thm: geodesic speed and distance} and Lemma~\ref{lem: measure on contact set} the distances are given by 
   \[
   d_{p}^{p}(v_{0}, v_{1}) = \int_{X} |\dot{v}_{0}|^{p}\eta^{n}_{v_{0}} = \int_{\{ v_{0} = f_{0}\}} |\dot{v}_{0}|^{p}\eta^{n}_{f_{0}}
   \]
   and 
   \[
   d_{p}^{p}(u_{0},u_{1}) = \int_{X} |\dot{u}_{0}|^{p}\te^{n}_{u_{0}} = \int_{\{u_{0} = f_{0}\}} |\dot{u}_{0}|^{p}\te^{n}_{f_{0}}.
   \]
   where $v_{t}$ is the $\eta$-geodesic joining $v_{0}$ and $v_{1}$, and $u_{t}$ is the $\te$-geodesic joining $u_{0}$ and $u_{1}$. Since $\te \leq \eta$, $u_{0} \leq v_{0}$ and $u_{1} \leq v_{1}$, $u_{t}$ is the $\eta$-subgeodesic joining $v_{0}$ and $v_{1}$, therefore, $u_{t} \leq v_{t}$. 

   Since $u_{0} \leq v_{0} \leq f_{0}$, the set $\{u_{0} = f_{0}\} \subset \{v_{0} = f_{0}\}$. Moreover, when $x \in \{u_{0} = f_{0}\}$, 
   \[
   \dot{u}_{0}(x) = \lim_{t \to 0} \frac{u_{t}(x) -u_{0}(x)}{t} = \lim_{t \to 0}  \frac{u_{t}(x) - f_{0}(x)}{t} \leq \lim_{t\to 0} \frac{v_{t}(x) - f_{0}(x)}{t} = \lim_{t \to 0} \frac{v_{t}(x) - v_{0}(x)}{t} = \dot{v}_{0}(x).
   \]
   Due to the assumption that $f_{0} \leq f_{1}$, $0 \leq \dot{u}_{0}, \dot{v}_{0}$, thus we obtain that $|\dot{u}_{0}| \leq |\dot{v}_{0}|$. Again due to $\te \leq \eta$, we have $\te^{n}_{f_{0}} \leq \eta^{n}_{f_{0}}$. Combining these we get 
   \[
   \int_{\{u_{0} = f_{0}\}} |\dot{u}_{0}|^{p}\te^{n}_{f_{0}} \leq \int_{\{ v_{0} = f_{0}\}} |\dot{v}_{0}|^{p}\eta^{n}_{f_{0}}.
   \]
   Therefore, $d_{p}(u_{0}, u_{1}) \leq d_{p}(v_{0},v_{1})$. 

   More generally, if $f_{0}$ and $f_{1}$ are unrelated, we can use the Pythagorean identity (see Lemma~\ref{lem: Pythagorean identity}) to prove the general result. We can do this because of \cite[Lemma 4.1]{guptacompletegeodescimetricinbigclasses} implies that $P_{\te}(u_{0}, u_{1}) \in \cH_{\te}$ and the fact that $P_{\te}(P_{\eta}(v_{0},v_{1})) = P_{\te}(u_{0},u_{1})$. 
   \begin{align*}
   d_{p}^{p}(u_{0},u_{1}) &= d_{p}^{p}(u_{0}, P_{\te}(u_{0},u_{1})) + d_{p}^{p}(u_{1}, P_{\te}(u_{0},u_{1})) \\
   &\leq d_{p}^{p}(v_{0}, P_{\eta}(v_{0},v_{1})) + d_{p}^{p}(v_{1}, P_{\eta}(v_{0},v_{1}))\\
   &= d_{p}^{p}(v_{0}, v_{1}).
   \end{align*}
\end{proof}

\section{Lidskii type inequality} \label{sec: Lidskii type inequality}  In this section we will prove Theorem~\ref{thm: introduction lidskii type inequality}. The proof will follow the approximation process used to construct the metric $d_{p}$ on $\E^{p}(X,\te)$. First, we will prove the Lidskii-type inequality for the big and nef cohomology classes and then prove it for the big cohomology classes.

\begin{lem} \label{lem: Lidskii inequality for big and nef classes}
    If $\bb$ represents a big and nef cohomology class and $u,v,w \in \E^{p}(X,\bb)$ satisfy $u \geq v \geq w$, then 
    \[
    d_{p}^{p}(v,w) \leq d_{p}^{p}(u,w) - d_{p}^{p}(u,v).
    \]
\end{lem}

\begin{proof}
    First, we assume that  $u,v, w \in \cH_{\bb}$. Let $\om_{\ee} = \bb + \ee\om$ represent a K\"ahler class. Let $u = P_{\bb}(f)$, $v = P_{\bb}(g)$, and $w = P_{\bb}(h)$ where $f,g,h \in C^{1,\bar{1}}(X)$ satisfy $f \geq g \geq h$. We denote by  $u_{\ee} = \Poe(f)$, $v_{\ee} = \Poe(g)$, and $w_{\ee} = \Poe(h)$. Since $u_{\ee} \geq v_{\ee} \geq w_{\ee}$ are in $\E^{p}(X,\om_{\ee})$, by \cite[Corollary 3.2]{DarvasLuuniformconvexity}, 
    \[
    d_{p}^{p}(v_{\ee}, w_{\ee}) \leq d_{p}^{p}(u_{\ee}, w_{\ee}) - d_{p}^{p}(u_{\ee}, v_{\ee}).
    \]
    \sloppy
    By the definition of $d_{p}$ on $\cH_{\bb}$ (see Section~\ref{sec: dp metric in big and nef}), we know that $\lim_{\ee \to 0} d_{p}(u_{\ee}, v_{\ee}) = d_{p}(u,v)$, $\lim_{\ee \to 0} d_{p}(u_{\ee}, w_{\ee}) = d_{p}(u,w)$, and $\lim_{\ee \to 0} d_{p}(v_{\ee}, w_{\ee}) = d_{p}(v,w)$. Thus taking the limit $\ee \to 0$ in the above equation we get 
    \[
    d_{p}^{p}(v,w) \leq d_{p}^{p}(u,w) - d_{p}^{p}(u,v).
    \]

    More generally, if $u,v,w \in \E^{p}(X,\bb)$, then we can find $u^{j}, v^{j}, w^{j} \in \cH_{\bb}$ such that $u^{j} \geq v^{j} \geq w^{j}$ and $u^{j} \searrow u$, $v^{j} \searrow v$, and $w^{j} \searrow w$ as $j \to \infty$. We just proved that 
    \[
    d_{p}^{p}(v^{j}, w^{j}) \leq d_{p}^{p}(u^{j}, w^{j}) - d_{p}^{p}(u^{j},v^{j}).
    \]
    Taking the limit $j \to \infty$ and using Lemma~\ref{lem: continuity under decreasing sequences}, we get 
    \[
    d_{p}^{p}(v,w) \leq d_{p}^{p} (u,w) - d_{p}^{p}(u,v).
    \]
    \end{proof}

    Before we can prove the Lidskii-type inequality for the big cohomology class, we need to show the result for the analytic singularity type. 
    \begin{lem}\label{lem: lidskii inequality for analytic singularity}
        If $\te$ represents a big cohomology class and $\psi$ is a $\te$-psh function with analytic singularity type, then for $u,v,w \in \E^{p}(X,\te, \psi)$ satisfying $u \geq v \geq w$, we have
        \[
        d_{p}^{p}(v,w) \leq d_{p}^{p}(u,w) - d_{p}^{p}(u,v).
        \]
    \end{lem}

    \begin{proof}
        Let $\mu : \tX \to X$ be the modification that resolves the singularities of $\psi$. As discussed in Section~\ref{sec: dp metric in analytic singularity} there is a smooth closed real $(1,1)$-form $\tte$ representing a big and nef cohomology class, and a bounded $\tte$-psh function $g$ on $\tX$ such that $\E^{p}(X,\te, \psi) \ni u \mapsto \tilde{u} = (u-\psi) \circ \mu + g \in \E^{p}(\tX, \tte)$ is a bijection. 

        Since $\tilde{u}, \tilde{v}, \tilde{w} \in \E^{p}(\tX, \tte)$ satisfy $\tilde{u} \geq \tilde{v} \geq \tilde{w}$, by Lemma~\ref{lem: Lidskii inequality for big and nef classes} we have 
        \[
        d_{p}^{p}(\tilde{v}, \tilde{w}) \leq d_{p}^{p} (\tilde{u}, \tilde{w}) - d_{p}^{p}(\tilde{u}, \tilde{v}).
        \]
        By Equation~\eqref{eq: equation for dp in the new paper}, we can write
        \[
        d_{p}^{p}(v,w) \leq d_{p}^{p}(u,w) - d_{p}^{p}(u,v).
        \]
    \end{proof}

    Now we can prove the Lidskii-type inequality for the big classes.
    \begin{thm} \label{thm: Lidskii inequality for big classes}
        Let $\te$ represent a big cohomology class. If $u,v,w \in \E^{p}(X,\te)$ satisfy $u \geq v \geq w$, then 
        \[
        d_{p}^{p}(v,w) \leq d_{p}^{p}(u,w) - d_{p}^{p}(u,v).
        \]
    \end{thm}

    \begin{proof}
        As discussed in Section~\ref{sec: dp metric in big}, let $\psi_{k}$ be an increasing sequence of $\te$-psh functions with analytic singularities such that $\psi_{k} \nearrow V_{\te}$ a.e. We denote by $u^{k} = P_{\te}[\psi_{k}](u)$, $v^{k} = P_{\te}[\psi_{k}](v)$, and $w^{k} = P_{\te}[\psi_{k}](w)$. Since $u \geq v \geq w$, we have $u^{k} \geq v^{k} \geq w^{k}$. By Lemma~\ref{lem: lidskii inequality for analytic singularity}, 
        \[
        d_{p}^{p}(v^{k}, w^{k}) \leq d_{p}^{p} (u^{k}, w^{k}) - d_{p}^{p}(u^{k}, v^{k}).
        \]
        Notice that the metric $d_{p}$ applied to $u^{k}, v^{k}, w^{k}$ is the metric on the space $\E^{p}(X,\te, \psi_{k})$, whereas the metric applied to $u,v,w$ is the metric on the space $\E^{p}(X,\te)$.
        Taking the limit $k \to \infty$ and using Equation~\eqref{eq: definition of dp for big class} we get 
        \[
        d_{p}^{p}(v,w) \leq d_{p}^{p}(u,w) - d_{p}^{p}(u,v).
        \]
    \end{proof}

\section{Buseman Convexity of $\E^{p}(X,\te)$} \label{sec: convexity}

In this section, we will prove Buseman convexity of the metric space $(\E^{p}(X,\te), d_{p})$ for $p \geq 1$ where $\te$ represents a big cohomology class. This means that we will prove that if $u_{0}, u_{1}, v_{0}, v_{1} \in \E^{p}(X,\te)$, and $u_{t}$ and $v_{t}$ are geodesic joining $u_{0},u_{1}$, and $v_{0},v_{1}$ respectively, then 
\[
d_{p}(u_{t},v_{t}) \leq (1-t)d_{p}(u_{0},v_{0}) + td_{p}(u_{1},v_{1}).
\]

\subsection{Buseman convexity for $p = 1$}\label{sec: buseman convexity for p = 1}
The Buseman convexity for $p =1$ follows from the arguments in \cite[Proposition 5.1]{BDL17}. We reproduce the arguments here for completeness. 

\begin{thm}\label{thm: Buseman convexity for p =1}
    If $\te$ represents a big cohomology class, $u_{0}, u_{1}, v_{0}, v_{1} \in \E^{p}(X,\te)$, and $u_{t}$ and $v_{t}$ are the weak geodesics joining $u_{0}$,$u_{1}$, and $v_{0}$, $v_{1}$ respectively, then 
    \[
    d_{1}(u_{t}, v_{t}) \leq (1-t)d_{1}(u_{0},v_{0}) + td_{1}(u_{1},v_{1}).
    \]
\end{thm}

\begin{proof}
    Fix $a\leq  b \in [0,1]$. Let $w_{a} := P_{\te}(u_{a}, v_{a})$ and $w_{b} = P_{\te}(u_{b}, v_{b})$. Let $[a,b] \ni t \mapsto w_{t} \in \E^{p}(X,\te)$ be the geodesic joining $w_{a}$ and $w_{b}$. Since $w_{a} \leq u_{a}, v_{a}$ and $w_{b} \leq u_{b}, v_{b}$, the comparison principle for geodesics implies that $w_{t} \leq u_{t}, v_{t}$ for all $ t \in [a,b]$. Thus $w_{t} \leq P_{\te}(u_{t}, v_{t})$. By monotonicity of $I$ (see \cite[Theorem 2.4 and Proposition 2.5]{darvas2021l1}, we get that $I(w_{t}) \leq I(P_{\te}(u_{t},v_{t})) $. As $I$ is linear along geodesics, we have  $I(P_{\te}(u_{t}, v_{t})) \geq I(w_{t}) = \frac{b-t}{b-a}I(P_{\te}(u_{a},v_{a})) + \frac{t-a}{b-a}I(P_{\te}(u_{b},v_{b}))$ for all $t \in [a,b]$. Thus $t \mapsto I(P_{\te}(u_{t},v_{t}))$ is a concave map. 

    By \cite[Theorem 5.7]{guptacompletegeodescimetricinbigclasses}, $d_{1}(u_{t}, v_{t}) = I(u_{t}) + I(v_{t}) - 2I(P_{\te}(u_{t},v_{t})$, thus it agrees with the metric introduced in \cite{darvas2021l1}. As $t \mapsto I(u_{t})$ and $t \mapsto I(v_{t})$ are linear and $t \mapsto I(P_{\te}(u_{t},v_{t}))$ is concave, we obtain that $t \mapsto d_{1}(u_{t},v_{t})$ is convex. 
\end{proof}

Thus, we only need to prove Buseman convexity for $p > 1$. This is great because now we can use uniform convexity of $(\E^{p}(X,\te), d_{p})$ from Theorem~\ref{thm: uniform convexity}. First, we will prove Buseman convexity for the big and nef case, and then prove the Buseman convexity for the big classes.

\subsection{Buseman convexity for big and nef classes when $p >  1$} \label{sec: buseman convexity for big and nef}
Fix a smooth closed real $(1,1)$-form $\bb$ representing a big and nef cohomology class and fix a $ p > 1$.  Recall from Section~\ref{sec: dp metric in big and nef} that $\om_{\ee} = \bb + \ee \om$ represents a K\"ahler class. Let $u_{0}, u_{1}, v_{0}, v_{1} \in \cH_{\bb}$ be given by $u_{0} = P_{\bb}(f_{0})$, $u_{1} = P_{\bb}(f_{1})$, $v_{0} = P_{\bb}(g_{0})$, and $v_{1} = P_{\bb}(g_{1})$ where $f_{0}, f_{1}, g_{0}, g_{1} \in C^{1,\bar{1}}(X)$. Let $u_{0,\ee} = \Poe(f_{0})$, $u_{1,\ee} = \Poe(f_{1})$, $v_{0,\ee} = \Poe(g_{0})$, and $v_{1,\ee} = \Poe(g_{1})$. Let $u_{t,\ee}$ and $v_{t,\ee}$ be the $\om_{\ee}$-geodesic joining $u_{0,\ee},u_{1,\ee}$, and $v_{0,\ee},v_{1,\ee}$ respectively. 

By the Buseman convexity in the K\"ahler case, we obtain that 
\[
d_{p}(u_{t,\ee}, v_{t,\ee}) \leq (1-t) d_{p}(u_{0,\ee}, v_{0,\ee}) + t d_{p}(u_{1,\ee}, v_{1,\ee}).
\]

\sloppy
By construction of $d_{p}$ on $\E^{p}(X,\bb)$, we know that $\lim_{\ee \to 0} d_{p}(u_{0,\ee}, v_{0,\ee}) = d_{p}(u_{0},v_{0})$ and $\lim_{\ee \to 0} d_{p}(u_{1,\ee}, v_{1,\ee}) = d_{p}(u_{1},v_{1})$. Thus to show Buseman convexity for potentials in $\cH_{\bb}$, it is enough to show that 
\[
d_{p}(u_{t},v_{t}) \leq \lim_{\ee \to 0} d_{p}(u_{t,\ee}, v_{t,\ee}),
\]
where $u_{t}$ and $v_{t}$ are the $\bb$-geodesics joining $u_{0}, u_{1}$ and $v_{0}, v_{1}$ respectively.

\begin{lem} \label{lem: projection getting closer property from Kahler to nef}
    Let $u_{0}, u_{1} \in \cH_{\bb}$ and $u_{t}$ be the geodesic joining them. Let $u_{0,\ee}, u_{1,\ee}$ be as above and $u_{t,\ee}$ be the geodesic joining $u_{0,\ee}$ and $u_{1,\ee}$. Then for $p > 1$, 
    \[
    \lim_{\ee \to 0} d_{p}(u_{t}, P_{\bb}(u_{t,\ee})) = 0.
    \]
\end{lem}

\begin{proof}

    We claim that $P_{\bb}(u_{t,\ee}) \searrow u_{t}$.  

    First notice that $u_{0, \ee} \searrow u_{0}$ and $u_{1,\ee} \searrow u_{1}$. Recall that by definition, $u_{0,\ee} = \Poe(f_{0})$ and $u_{0} = P_{\bb}(f_{0})$ for $f_{0} \in C^{1,\bar{1}}$. If $\ee_{1} < \ee_{2}$, then $u_{0,\ee_{1}}$ is $\om_{\ee_{2}}$-psh as well and $u_{0,\ee_{1}} \leq f_{0}$, therefore, $u_{0,\ee_{1}} \leq u_{0,\ee_{2}}$. By a similar argument $u_{0,\ee} \geq u_{0}$ as well. Therefore, $u_{0,\ee}$ are increasing sequence of functions, and $\lim_{\ee \to 0} u_{0,\ee} \geq u_{0}$. The limit $\lim_{\ee \to 0}u_{0,\ee}$ is a $\bb + \ee\om$-psh for every $\ee > 0$, therefore, $\lim_{\ee \to 0} u_{0,\ee}$ is a $\bb$-psh function that satisfies $\lim_{\ee \to 0} u_{0,\ee} \leq f_{0}$, thus $\lim_{\ee \to 0} u_{0,\ee} \leq u_{0}$. Thus $\lim_{\ee \to 0 } u_{0,\ee} = u_{0}$.   Similarly, $u_{1.\ee} \searrow u_{1}$. 

    We can make the same argument for the $\om_{\ee}$-geodesics $u_{t,\ee}$ joining $u_{0,\ee}$ and $u_{1,\ee}$. If $\ee_{1} \leq \ee_{2}$, then the path $(0,1) \ni t \mapsto u_{t,\ee_{1}}$ is an $\om_{\ee_{2}}$-subgeodesic joining $u_{0,\ee_{2}}$ and $u_{1,\ee_{2}}$. Therefore, $u_{t,\ee_{1}} \leq u_{t,\ee_{2}}$. Similarly, $u_{t,\ee} \geq u_{t}$. Therefore, $(0,1) \ni t \mapsto u_{t,\ee}$  is a decreasing sequence of paths and $\lim_{\ee \to 0} u_{t,\ee} \geq u_{t}$. The limit path $(0,1) \ni t \mapsto \lim_{\ee \to 0} u_{t,\ee}$ is a $\bb + \ee\om$-subgeodesic joining $u_{0}$ and $u_{1}$ for each $\ee > 0$, therefore it is a $\bb$-subgeodesic as well. Hence $\lim_{\ee \to 0} u_{t,\ee} \leq u_{t}$. So we get $\lim_{\ee \to 0} u_{t,\ee} = u_{t}$. 

    Since $u_{t,\ee} \searrow u_{t}$ and $u_{t,\ee} \geq P_{\bb}(u_{t,\ee}) \geq u_{t}$, we have that $P_{\bb}(u_{t,\ee}) \searrow u_{t}$ as well. Thus proving the claim. Now the proof of the lemma follows from Lemma~\ref{lem: continuity under decreasing sequences}.

\end{proof}

Similarly for $p > 1$, $\lim_{\ee \to 0}d_{p}(v_{t}, P_{\bb}(v_{t,\ee})) = 0$. By the triangle inequality, 
\[
\lim_{\ee \to 0} d_{p}(P_{\bb}(u_{t,\ee}), P_{\bb}(v_{t,\ee})) = d_{p}(u_{t}, v_{t}).
\]

 Since $f_{0}, f_{1} \in C^{1,\bar{1}}(X)$,  Lemma~\ref{lem: envelope for eta} implies that $u_{0,\ee},u_{1,\ee} \in C^{1,\bar{1}}(X)$. Thus Lemma~\ref{lem: C11bar regularity for eta} implies that the geodesic $u_{t,\ee} \in C^{1,\bar{1}}(X)$ as well. Similarly, the geodesic $v_{t,\ee} \in C^{1,\bar{1}}(X)$. Thus we can apply Lemma~\ref{lem: a general contraction property}, to get $d_{p}(P_{\bb}(u_{t,\ee}), P_{\bb}(v_{t,\ee})) \leq d_{p}(u_{t,\ee}, v_{t,\ee})$. Combining this with the above equation we get 
\[
d_{p}(u_{t}, v_{t}) = \lim_{\ee \to 0} d_{p}(P_{\bb}(u_{t,\ee}), P_{\bb}(v_{t,\ee})) \leq \lim_{\ee \to 0} d_{p}(u_{t,\ee}, v_{t, \ee})
\]
as we needed to show. Thus we have proved
\begin{lem}\label{lem: buseman convexity for big and nef in cHbb}
    Let  $\bb$ represent a big and nef cohomology class and $p > 1$. Let  $u_{0}, u_{1}, v_{0}, v_{1} \in \cH_{\bb}$. If $u_{t}$ and $v_{t}$ are the geodesics joining $u_{0}, u_{1}$ and $v_{0}, v_{1}$ respectively, then 
    \[
    d_{p}(u_{t}, v_{t}) \leq (1-t) d_{p}(u_{0},v_{0}) + td_{p}(u_{1},v_{1}). 
    \]
\end{lem}

Now we will extend this result to $\E^{p}(X,\bb)$ for $p > 1$. 

\begin{thm}\label{thm: buseman convexity for big and nef and p > 1}
    Let  $\bb$ represent a big and nef cohomology class and $p > 1$. Let  $u_{0}, u_{1}, v_{0}, v_{1} \in \E^{p}(X,\bb)$. If $u_{t}$ and $v_{t}$ are the geodesics joining $u_{0}, u_{1}$ and $v_{0}, v_{1}$ respectively, then 
    \[
    d_{p}(u_{t}, v_{t}) \leq (1-t) d_{p}(u_{0},v_{0}) + td_{p}(u_{1},v_{1}). 
    \]
\end{thm}
\begin{proof}
    As standard, we can find $u_{0}^{j}, u_{1}^{j}, v_{0}^{j}, v_{1}^{j} \in \cH_{\bb}$ such that $u_{0}^{j} \searrow u_{0}$, $u_{1}^{j} \searrow u_{1}$, $v_{0}^{j} \searrow v_{0}$, and $v_{1}^{j} \searrow v_{1}$. If $u_{t}^{j} $ and $v_{t}^{j}$ are the geodesics joining $u_{0}^{j}$, $u_{1}^{j}$, and $v_{0}^{j}$, $v_{1}^{j}$ respectively, then $u_{t}^{j} \searrow u_{t}$ and $v_{t}^{j} \searrow v_{t}$. 

    By the definition of $d_{p}$ for $\E^{p}(X,\bb)$ in Section~\ref{sec: dp metric in big and nef}, we know that $\lim_{j\to \infty} d_{p}(u_{0}^{j}, v_{0}^{j}) = d_{p}(u_{0}, v_{0})$ and that $\lim_{j \to \infty} d_{p}(u_{1}^{j}, v_{1}^{j}) = d_{p}(u_{1}, v_{1})$. Moreover, due to Lemma~\ref{lem: continuity under decreasing sequences}, $\lim_{j \to \infty} d_{p}(u_{t}^{j}, v_{t}^{j}) = d_{p}(u_{t}, v_{t}) $. By Lemma~\ref{lem: buseman convexity for big and nef in cHbb}, we have 
    \[
    d_{p}(u_{t}^{j}, v_{t}^{j}) \leq (1-t)d_{p}(u_{0}^{j}, v_{0}^{j}) + t d_{p}(u_{1}^{j}, v_{1}^{j}).
    \]
    Taking limit $j \to \infty$, we get that 
    \[
    d_{p}(u_{t}, v_{t}) \leq (1-t) d_{p}(u_{0}, v_{0}) + t d_{p}(u_{1}, v_{1}). 
    \]
\end{proof}

\subsection{Buseman Convexity for big classes when $p > 1$} 

As in the case of Lidskii-type inequality in Section~\ref{sec: Lidskii type inequality}, we will first show the Buseman convexity for the analytic singularity case, and then for the minimal singularity case for the big classes. 

Assume that $\te$ is a smooth closed real $(1,1)$-form representing a big cohomology class. Let $\psi$ be a $\te$-psh function with analytic singularities. We will first show that $(\E^{p}(X,\te,\psi), d_{p})$ is Buseman convex. 

\begin{lem}\label{lem: Buseman convexity for analytic singularity type and p > 1}
 For $p > 1$, the metric space $(\E^{p}(X,\te, \psi), d_{p})$ is a Buseman convex metric space. This means that if $u_{0}, u_{1}, v_{0}, v_{1} \in \E^{p}(X,\te, \psi)$, and $u_{t}$ and $v_{t}$ are the weak geodesics joining $u_{0}$,$u_{1}$, and $v_{0}$, $v_{1}$ respectively, then 
    \[
    d_{p}(u_{t}, v_{t}) \leq (1-t)d_{p}(u_{0},v_{0}) + td_{p}(u_{1},v_{1}).
    \]
\end{lem}

\begin{proof}
    Recall from Section~\ref{sec: dp metric in analytic singularity} that there exists a modification $\mu : \tX \to X$, a big and nef cohomology class $\tte$ and a bounded $\tte$-psh function $g$ on $\tX$ such that the metric $d_{p}$ on $\E^{p}(X,\te,\psi)$ is given by 
    \[
    d_{p}(u_{0}, u_{1}) = d_{p}(\tilde{u}_{0}, \tilde{u}_{1})
    \]
    where $\E^{p}(X,\te,\psi) \ni u \mapsto \tilde{u}:= (u-\psi)\circ \mu + g  \in \E^{p}(\tX,\tte)$ is a bijection. 

    By Theorem~\ref{thm: buseman convexity for big and nef and p > 1}, the metric space $(\E^{p}(\tX,\tte), d_{p})$ is Buseman convex. Therefore, 
    \[
    d_{p}(u_{t}, v_{t}) = d_{p}(\tilde{u}_{t}, \tilde{v}_{t}) \leq (1-t)d_{p}(\tilde{u}_{0}, \tilde{v}_{0}) + t d_{p}(\tilde{u}_{1}, \tilde{v}_{1}) = (1-t) d_{p}(u_{0}, v_{0})  + t d_{p}(u_{1}, v_{1}). 
    \]
\end{proof}

Now we can show that the metric space $(\E^{p}(X,\te),d_{p})$ is Buseman convex for $p > 1$. 

\begin{thm} \label{thm: Buseman convexity for big class and p > 1}
    For $p > 1$, the metric space $(\E^{p}(X,\te), d_{p})$ is a Buseman convex metric space. This means that if $u_{0}, u_{1}, v_{0}, v_{1} \in \E^{p}(X,\te)$, and $u_{t}$ and $v_{t}$ are the weak geodesics joining $u_{0}$,$u_{1}$, and $v_{0}$, $v_{1}$ respectively, then 
    \[
    d_{p}(u_{t}, v_{t}) \leq (1-t)d_{p}(u_{0},v_{0}) + td_{p}(u_{1},v_{1}).
    \]
\end{thm}

\begin{proof}
    As described in Section~\ref{sec: dp metric in big}, let $\psi_{k}$ be an increasing sequence of $\te$-psh functions with analytic singularities such $\psi_{k} \nearrow V_{\te}$ a.e.

    We denote the projections of $u_{0},u_{1},v_{0}$, and $v_{1}$ to $\E^{p}(X,\te,\psi_{k})$ by $u_{0}^{k} = P_{\te}[\psi_{k}](u_{0})$, $u_{1}^{k} = P_{\te}[\psi_{k}](u_{1})$, $v_{0}^{k} = P_{\te}[\psi_{k}](v_{0})$, and $v_{1}^{k} = P_{\te}[\psi_{k}](v_{1})$ respectively. We denote by $u_{t}^{k}$ and $v_{t}^{k}$ the weak geodesics joining $u_{0}^{k}$, $u_{1}^{k}$, and $v_{0}^{k}$, $v_{1}^{k}$ respectively. 

    By a proof similar to Lemma~\ref{lem: projection getting closer property from Kahler to nef} we can show that 
    \begin{equation}\label{eq: projection coming closer from analytic singularity to big}
    \lim_{k \to \infty} d_{p}(u_{t}^{k}, P_{\te}[\psi_{k}](u_{t})) = 0 \qquad \text{ and } \qquad  \lim_{k\to \infty} d_{p}(v_{t}^{k}, P_{\te}[\psi_{k}](v_{t})) = 0.     
    \end{equation}
    See the proof of \cite[Theorem 8.3, Equation 26]{guptacompletegeodescimetricinbigclasses} for more details. 

    From Equation~\eqref{eq: definition of dp for big class} and~\ref{eq: projection coming closer from analytic singularity to big}, we get
    \[
    d_{p}(u_{t}, v_{t}) = \lim_{k \to \infty} d_{p}(P_{\te}[\psi_{k}](u_{t}), P_{\te}[\psi_{k}](v_{t})) = \lim_{k\to \infty}d_{p}(u_{t}^{k}, v_{t}^{k}). 
    \]
    By Lemma~\ref{lem: Buseman convexity for analytic singularity type and p > 1}, we can write 
    \[
    d_{p}(u_{t}^{k}, v_{t}^{k}) \leq (1-t)d_{p}(u_{0}^{k}, v_{0}^{k}) + t d_{p}(u_{1}^{k}, v_{1}^{k}).
    \]
    Combining the last two equations we get 
    \[
    d_{p}(u_{t}, v_{t}) = \lim_{k \to \infty}d_{p}(u_{t}^{k}, v_{t}^{k}) \leq \lim_{k \to \infty} (1-t)d_{p}(u_{0}^{k}, v_{0}^{k}) + t d_{p}(u_{1}^{k}, v_{1}^{k}) = (1-t)d_{p}(u_{0}, v_{0}) + t d_{p}(u_{1}, v_{1}). 
    \]
\end{proof}

Combined with Theorem~\ref{thm: Buseman convexity for p =1}, this finishes the proof of Theorem~\ref{thm: introduction buseman convexity}.

\section{Space of geodesic rays}\label{sec: space of geodesic rays}

In this section, we will study the space of geodesic rays and prove Theorem~\ref{thm: Introduction finite energy geodesic ray}. Let $u \in \E^{p}(X,\te)$. A geodesic ray $\{u_{t}\}_{t}$ is a finite $p$-energy geodesic ray emanating from $u$ if $u_{t} \to u$ as $t \to 0$ in $L^{1}(X)$ and $u_{t} \in \E^{p}(X,\te)$ for all $t> 0$. We denote by $\cR^{p}_{u}$ the set of all finite $p$-energy geodesic rays emanating from $u$. 

We construct a metric $d_{p}^{c}$ on $\cR^{p}_{u}$ as follows. If $\{u_{t}\}_{t}, \{v_{t}\}_{t} \in \cR^{p}_{u}$ are two finite energy geodesic rays, then we define
\begin{equation}\label{eq: equation for distance between rays}
d_{p}^{c}(\{u_{t}\}_{t}, \{v_{t}\}_{t}) = \lim_{t \to \infty} \frac{d_{p}(u_{t}, v_{t})}{t}.    
\end{equation}

Due to Buseman convexity of $d_{p}$ from Theorem~\ref{thm: introduction buseman convexity}, we know that if $t_{1} \geq t_{2}$, then 
\[
d_{p}(u_{t_{2}}, v_{t_{2}}) \leq \left(1 -\frac{t_{2}}{t_{1}}\right) d_{p}(u_{0}, v_{0}) + \frac{t_{2}}{t_{1}}d_{p}(u_{t_{1}}, v_{t_{1}}). 
\]
Since $u_{0} = v_{0} = u$, we have 
\[
\frac{d_{p}(u_{t_{2}}, v_{t_{2}})}{t_{2}} \leq \frac{d_{p}(u_{t_{1}}, v_{t_{1}})}{t_{1}}.
\]
Thus the limit in Equation~\eqref{eq: equation for distance between rays} is increasing and is bounded from above due to the triangle inequality. This is because $d_{p}(u_{t}, v_{t}) \leq d_{p}(u_{t}, u) + d_{p}(v_{t}, u) = td_{p}(u_{1}, u) + t d_{p}(v_{1}, u)$. Thus 
\[
\frac{d_{p}(u_{t}, v_{t})}{t} \leq d_{p}(u_{1},u) + d_{p}(v_{1},u).
\]
Thus the limit in Equation~\eqref{eq: equation for distance between rays} is well defined. 

\begin{lem}\label{lem: metric on space of rays}
    Equation~\eqref{eq: equation for distance between rays} defines a metric on $\cR^{p}_{u}$. 
\end{lem}

\begin{proof}
    Let $\{u_{t}\}_{t}, \{v_{t}\}_{t}, \{w_{t}\}_{t} \in \cR^{p}_{u}$ be three finite $p$-energy geodesic rays. $d_{p}^{c} (\{u_{t}\}_{t}, \{v_{t}\}_{t}) \geq 0$ because $d_{p}(u_{t}, v_{t}) \geq 0$. 

    If $d_{p}^{c}(\{u_{t}\}_{t}, \{v_{t}\}_{t}) = 0$, then the function $f(t) = d_{p}(u_{t}, v_{t})/t$ is an increasing function that satisfies $f(0) = 0$ and $\lim_{t\to\infty}f(t) = 0$. Therefore $f(t) = 0$ for all $t \geq 0$. Thus $d_{p}(u_{t}, v_{t}) = 0$ for all $t \geq 0$. Therefore $u_{t} = v_{t}$ for all $t \geq 0$. So $\{u_{t}\}_{t} = \{v_{t}\}_{t}$. This proves non-degeneracy of $d_{p}^{c}$.

    The triangle inequality for $d_{p}^{c}$ follows from the triangle inequality for $d_{p}$. 
    \begin{align*}
d_{p}^{c}(\{u_{t}\}_{t}, \{v_{t}\}_{t}) &= \lim_{t\to\infty} \frac{d_{p}(u_{t},v_{t})}{t}\\
&\leq \lim_{t\to\infty}\frac{d_{p}(u_{t}, w_{t})}{t} + \lim_{t\to\infty}\frac{d_{p}(w_{t},v_{t})}{t}\\
&= d_{p}^{c}(\{u_{t}\}_{t}, \{w_{t}\}_{t}) + d_{p}^{c}(\{w_{t}\}_{t}, \{v_{t}\}_{t}).        
    \end{align*}

\end{proof}

The next theorem shows that the metric space $(\cR^{p}_{u}, d_{p}^{c})$ are all isometric for all $u \in \E^{p}(X,\te)$. 

\begin{thm} \label{thm: projection operator}
    Given $u,v \in \E^{p}(X,\te)$, there exists a map $P_{uv} : \cR^{p}_{u} \to \cR^{p}_{v}$ that induces a bijective isometry $P_{uv} : (\cR^{p}_{u}, d_{p}^{c}) \to (\cR^{p}_{v}, d_{p}^{c})$. 
\end{thm}

\begin{proof}
    Let $\{u_{t}\}_{t} \in \cR^{p}_{u}$. We will construct a finite energy geodesic ray $\{v_{t}\}_{t}$ starting from $v$. 

    First, assume that $u \geq v$. We can construct a finite energy geodesic $[0,t] \ni l \mapsto v_{l}^{t} \in \E^{p}(X,\te)$ joining $v = v_{0}^{t}$ and $u_{t} = v_{t}^{t}$. Consider $0 \leq t \leq t'$. Since $[0,t'] \ni l \mapsto u_{l}$ is a finite energy geodesic joining $u$ and $u_{t'}$, whereas $[0,t'] \ni l \mapsto v_{l}^{t'}$ is a finite energy geodesic joining $v$ and $u_{t'}$. Since $u \geq v$, by the comparison principle, $v_{t}^{t'} \leq u_{t} = v_{t}^{t}$. Again applying the comparison principle, we get that for $l \leq t$, $v_{l}^{t'} \leq v_{l}^{t}$. 

    Moreover, by Theorem~\ref{thm: introduction buseman convexity}, 
    \[
    d_{p}(v_{l}^{t}, u_{l}) \leq \left( 1- \frac{l}{t}\right) d_{p}(u,v).
    \]
    By Lemma~\ref{lem: endpoint stability}, the map $[0,\infty) \ni l \mapsto v_{l} := \lim_{t \to \infty}v_{l}^{t} \in \E^{p}(X,\te)$ is a finite energy geodesic ray. Note that the limit geodesic $v_{l}$ still emanates from $v$. This is because by Lemma~\ref{lem: endpoint stability}, the path $[0,t] \ni l \mapsto v_{l} \in \E^{p}(X,\te)$ is a geodesic joining $v$ and $v_{t}$. Because $v, v_{t} \in E^{p}(X,\te)$ the result follows from the following argument. As $d_{p}(v_{l}, v_{0}) = \frac{l}{t}$, we get that $\lim_{l \to 0} d_{p}(v_{l}, v_{0}) = 0$. Thus by \cite[Lemma 4.12]{guptacompletegeodescimetricinbigclasses} and \cite[Theorem 1.2]{GuptaCompletemetrictopology}, we get that $v_{l} \to v_{0}$ in $L^{1}(X)$. See \cite[Proposition 4.2.1]{minchennotes} for a more general result. Now, taking the limit $t \to \infty$ in the equation above we get that $d_{p}(v_{l}, u_{l}) \leq d_{p}(u,v)$. 

    A similar proof works when $u\leq v$. 

    More generally, we can consider $h = P_{\te}(u,v) \in \E^{p}(X,\te)$. Now $h \leq u, v$. Thus from above, we can construct a finite energy geodesic ray $\{h_{t}\}_{t} \in \cR^{p}_{h}$ that satisfies $d_{p}(h_{t}, u_{t}) \leq d_{p}(h,u)$. We can now also construct a finite energy geodesic ray $\{v_{t}\}_{t} \in \cR^{p}_{v}$ such that $d_{p}(h_{t}, v_{t}) \leq d_{p}(h,v)$. Thus by the triangle inequality $d_{p}(u_{t}, v_{t}) \leq d_{p}(u_{t},h_{t}) + d_{p}(h_{t},v_{t}) \leq d_{p}(u,h) + d_{p}(h,v)$. Thus $\{u_{t}\}_{t}, \{v_{t}\}_{t}$ are parallel geodesic rays. 

    To see that $P_{uv} : (\cR^{p}_{u}, d_{p}^{c}) \to (\cR^{p}_{v}, d_{p}^{c})$ is an isometry consider $\{u_{t}^{0}\}_{t}, \{u_{t}^{1}\}_{t} \in \cR^{p}_{u}$. Let $P_{uv}(\{u_{t}^{i}\}_{t}) = \{v_{t}^{i}\}_{t}$ for $i = 0,1$. By the triangle inequality we have 
    \begin{align*}
        d_{p}^{c}(\{v_{t}^{0}\}_{t}, \{v_{t}^{1}\}_{t}) &= \lim_{t\to \infty} \frac{d_{p}(v_{t}^{0}, v_{t}^{1})}{t} \\
        &\leq \lim_{t \to \infty} \frac{d_{p}(v_{t}^{0}, u_{t}^{0}) + d_{p}(u_{t}^{0}, u_{t}^{1}) + d_{p}(u_{t}^{1}, v_{t}^{1})}{t}\\
        &= \lim_{t\to \infty} \frac{d_{p}(u_{t}^{0}, u_{t}^{1})}{t}\\
        &= d_{p}^{c}(\{u_{t}^{0}\}_{t}, \{u_{t}^{1}\}_{t}).
    \end{align*}
    In the second line, we used that fact that $d_{p}(u_{t}^{i}, v_{t}^{i})$ remain bounded as $t \to \infty$ for $i = 0,1$. The other side inequality is obtained similarly. Thus $d_{p}^{c}(\{u_{t}^{0}\}_{t}, \{u_{t}^{1}\}_{t}) = d_{p}^{c}(\{v_{t}^{0}\}_{t}, \{v_{t}^{1}\}_{t})$.  
\end{proof}

Since the spaces $(\cR^{p}_{u}, d_{p}^{c})$ are all isometric to each other, it is good enough to study just one of them. We will fix $u = V_{\te}$ and use $\cR^{p}_{\te}$ to denote $\cR^{p}_{V_{\te}}$. The goal of the rest of the section is to prove that $(\cR^{p}_{\te}, d_{p}^{c})$ is a complete geodesic metric space. First, we prove

\begin{lem}
    The metric space $(\cR^{p}_{\te}, d_{p}^{c})$ is a complete metric space. 
\end{lem}
\begin{proof}
\sloppy
    Let $\{u_{t}^{j}\}_{t} \in \cR^{p}_{\te}$ be a Cauchy sequence of finite energy geodesic rays. This means that $d_{p}^{c}(\{u_{t}^{j}\}_{t}, \{u_{t}^{k}\}_{t}) \to 0$ as $j,k \to \infty$. For a fixed $t \geq 0$, we know that $d_{p}(u_{t}^{j}, u_{t}^{k}) \leq td_{p}^{c}(\{u_{t}^{j}\}_{t},\{u_{t}^{k}\}_{t})$. Thus $u_{t}^{j}$ is a Cauchy sequence in $(\E^{p}(X,\te), d_{p})$ with limit point $u_{t} \in \E^{p}(X,\te)$. By the endpoint stability Lemma~\ref{lem: endpoint stability}, $u_{t}$ is a finite energy geodesic ray. Now we observe that 
    \begin{align*}
        d_{p}^{c}(\{u_{t}^{j}\}_{t}, \{u_{t}\}_{t}) &= \lim_{t \to \infty}\frac{d_{p}(u_{t}^{j}, u_{t})}{t} \\
        &= \lim_{t \to \infty} \lim_{k \to \infty}\frac{d_{p}(u_{t}^{j}, u_{t}^{k})}{t} \\
        &\leq d_{p}^{c}(\{u_{t}^{j}\}_{t} , \{u_{t}^{k}\}_{t})
    \end{align*}
    which can be arbitrarily small. Therefore, $d_{p}^{c}(\{u_{t}^{j}\}_{t}, \{u_{t}\}_{t}) \to 0$ as $j \to \infty$. 
\end{proof}

First, we will show that for $p > 1$, $(\cR^{p}_{\te}, d_{p}^{c})$ is a complete geodesic metric space. We have already seen that it is a complete metric space. We just need to show that there are geodesics in this space. 

\begin{thm} \label{thm: geodesics chords for p> 1}
    For $p > 1$, $(\cR^{p}_{\te}, d_{p}^{c})$ is a complete geodesic metric space. 
\end{thm}

\begin{proof}
    Let $\{u_{t}^{0}\}_{t}, \{u_{t}^{1}\}_{t} \in \cR^{p}_{\te}$ be two finite energy geodesic rays. We will construct finite energy geodesic rays $\{u_{t}^{\al}\}_{t} \in \cR^{p}_{\te}$ for $0 \leq \al \leq 1$, such that the path $[0,1] \ni \al \mapsto \{u_{t}^{\al}\}_{t} \in \cR^{p}_{\te}$ is a $d_{p}^{c}$-geodesic. 

    First, consider the finite energy geodesic $[0,1]\ni \al \mapsto v_{t}^{\al} \in \E^{p}(X,\te)$ joining $u_{t}^{0}$ and $u_{t}^{1}$. Let $[0,t] \ni l \mapsto w_{l}^{\al,t} \in \E^{p}(X,\te)$ be the finite energy geodesic joining $w_{0}^{\al, t} := V_{\te}$ and $w_{t}^{\al,t} = v_{t}^{\al}$. We claim that for a fixed $l \geq 0$, $w_{l}^{\al,t} \in \E^{p}(X,\te)$ is $d_{p}$-Cauchy as $t \to \infty$. Thus there exist $u_{l}^{\al} \in \E^{p}(X,\te)$ such that $\lim_{t \to \infty} d_{p}(w_{l}^{\al,t}, u_{l}^{\al}) = 0$. By the endpoint stability Lemma~\ref{lem: endpoint stability}, $\{u_{t}^{\al}\}_{t}$ is a finite energy geodesic ray. 

    Now we will prove the claim. Let $s \leq t$. By Theorem~\ref{thm: introduction buseman convexity}, 
    \begin{equation}\label{eq: approximating the radial geodesic 1}
                \frac{d_{p}(u_{s}^{0}, w_{s}^{\al,t})}{s} \leq \frac{d_{p}(u_{t}^{0}, v_{t}^{\al})}{t} = \al\frac{d_{p}(u_{t}^{0}, u_{t}^{1})}{t} \leq \al d_{p}^{c}(\{u_{t}^{0}\}_{t}, \{u_{t}^{1}\}_{t}) 
    \end{equation}

    Similarly, we have 
    \begin{equation}\label{eq: approximating the radial geodesic 2}
            \frac{d_{p}(u_{s}^{1}, w_{s}^{\al,t})}{s} \leq \frac{d_{p}(u_{t}^{1}, v_{t}^{\al})}{t} = (1-\al)\frac{d_{p}(u_{t}^{0}, u_{t}^{1})}{t} \leq (1-\al) d_{p}^{c}(\{u_{t}^{0}\}_{t}, \{u_{t}^{1}\}_{t}) 
    \end{equation}

    Since $\frac{d_{p}(u_{s}^{0}, u_{s}^{1})}{s} \nearrow d_{p}^{c}(\{u_{t}^{0}\}_{t}, \{u_{t}^{1}\}_{t})$. We can write $d_{p}^{c}(\{u_{t}^{0}\}_{t}, \{u_{t}^{1}\}_{t}) \leq (1+\ee(s))\frac{d_{p}(u_{s}^{0}, u_{s}^{1})}{s}$ where $\ee(s) \to 0$ as $s \to \infty$. Combining with Equations~\eqref{eq: approximating the radial geodesic 1} and~\eqref{eq: approximating the radial geodesic 2}, we get 
    \[
    d_{p}(u_{s}^{0}, w_{s}^{\al,t}) \leq (\al + \ee(s)) d_{p}(u_{s}^{0}, u_{s}^{1})
    \]
    and
    \[
    d_{p}(u_{s}^{1}, w_{s}^{\al, t}) \leq (1-\al +\ee(s)) d_{p}(u_{s}^{0}, u_{s}^{1}).
    \]
    By Corollary~\ref{cor: uniform convexity} of the uniform convexity, we can write 
    \[
        d_{p}(v_{s}^{\al}, w_{s}^{\al,t}) \leq C\ee(s)^{\frac{1}{r}}d_{p}(u_{s}^{0}, u_{s}^{1}).
    \]
    This is the only place where we need the condition that $p > 1$. For $l \leq s$, using Theorem~\ref{thm: introduction buseman convexity},
    \[
    \frac{d_{p}(w_{l}^{\al,s}, w_{l}^{\al,t})}{l} \leq \frac{d_{p}(v_{s}^{\al}, w_{s}^{\al,t})}{s} \leq C(\ee(s))^{\frac{1}{r}}d_{p}^{c}(\{u_{t}^{0}\}_{t}, \{u_{t}^{1}\}_{t}).
    \]
    As $\ee(s)$ can be made arbitrarily small, we get that $w_{l}^{\al,t}$ is a $d_{p}$-Cauchy sequence as $t \to \infty$. By \cite[Main Theorem]{guptacompletegeodescimetricinbigclasses}, the metric $d_{p}$ is complete, so there exists $u_{l}^{\al} \in \E^{p}(X,\te)$ such that $\lim_{t \to \infty}d_{p}(w_{l}^{\al, t}, u_{l}^{\al}) = 0$.  

    Now we will prove that the map $[0,1] \ni \al \mapsto \{u_{t}^{\al}\}_{t} \in \cR^{p}_{\te}$ is a $d_{p}^{c}$-geodesic. Taking the limit $t \to \infty$ in Equations~\eqref{eq: approximating the radial geodesic 1} and~\eqref{eq: approximating the radial geodesic 2}, we get that 
    \[
    \frac{d_{p}(u_{s}^{0}, u_{s}^{\al})}{s} \leq \al d_{p}^{c}(\{u_{t}^{0}\}_{t}, \{u_{t}^{1}\}_{t}).
    \]
    and 
    \[
    \frac{d_{p}(u_{s}^{1}, u_{s}^{\al})}{s} \leq (1-\al)d_{p}^{c}(\{u_{t}^{0}\}_{t}, \{u_{t}^{1}\}_{t}).
    \]
    Taking the limit $s \to \infty$, we get that 
    \[
    d_{p}^{c}(\{u_{t}^{0}\}_{t}, \{u_{t}^{\al}\}_{t}) \leq \al d_{p}^{c}(\{u_{t}^{0}\}_{t}, \{u_{t}^{1}\}_{t}) \qquad \text{ and } \qquad  d_{p}^{c}(\{u_{t}^{1}\}_{t}, \{u_{t}^{\al}\}_{t}) \leq (1-\al) d_{p}^{c}(\{u_{t}^{0}\}_{t}, \{u_{t}^{1}\}_{t})
    \]
    Combining with the triangle inequality, we get 
    \[
    d_{p}^{c}(\{u_{t}^{0}\}_{t}, \{u_{t}^{\al}\}_{t}) = \al d_{p}^{c}(\{u_{t}^{0}\}_{t}, \{u_{t}^{1}\}_{t}) \qquad \text{ and } \qquad  d_{p}^{c}(\{u_{t}^{1}\}_{t}, \{u_{t}^{\al}\}_{t}) = (1-\al) d_{p}^{c}(\{u_{t}^{0}\}_{t}, \{u_{t}^{1}\}_{t})
    \]

    For $0 \leq \al \leq \bb \leq 1$, the triangle inequality implies 
    \[
    d_{p}^{c}(\{u_{t}^{\al}\}_{t} , \{u_{t}^{\bb}\}_{t}) \geq d_{p}^{c}(\{u_{t}^{0}\}_{t}, \{u_{t}^{\bb}\}_{t}) - d_{p}^{c}(\{u_{t}^{0}\}_{t}, \{u_{t}^{\al}\}_{t}) = (\bb- \al)d_{p}^{c}(\{u_{t}^{0}\}_{t}, \{u_{t}^{1}\}_{t}).
    \]
    For the other side, we again use Theorem~\ref{thm: introduction buseman convexity}. We notice that 
    \[
    \frac{d_{p}(w_{l}^{\al,t}, w_{l}^{\bb,t})}{l} \leq \frac{d_{p}(v_{t}^{\al}, v_{t}^{\bb})}{t} = (\bb-\al)\frac{d_{p}(u_{t}^{0}, u_{t}^{1})}{t}
    \]
    Taking limit $t \to \infty$, we get
    \[
    \frac{d_{p}(u_{l}^{\al}, u_{l}^{\bb})}{l} \leq (\bb-\al)d_{p}^{c}(\{u_{t}^{0}\}_{t}, \{u_{t}^{1}\}_{t}).
    \]
    Now taking limit $l \to \infty$, we get 
    \[
    d_{p}^{c}(\{u_{t}^{\al}\}_{t}, \{u_{t}^{\bb}\}_{t}) \leq (\bb - \al)d_{p}^{c}(\{u_{t}^{0}\}_{t}, \{u_{t}^{1}\}_{t}).
    \]
    Thus 
    \[
    d_{p}^{c}(\{u_{t}^{\al}\}_{t}, \{u_{t}^{\bb}\}_{t}) = (\bb - \al)d_{p}^{c}(\{u_{t}^{0}\}_{t}, \{u_{t}^{1}\}_{t})
    \]
    proving that the map $[0,1] \ni \al \mapsto \{u_{t}^{\al}\}_{t} \in \cR^{p}_{\te}$ is a $d_{p}^{c}$-geodesic.
\end{proof}

Now we need to extend this proof to $p = 1$ by approximating it from $p > 1$. First, we will show that the $d_{p}^{c}$-geodesics depend only on the end-points, and not the value of $p$, just like in the metric space $(\E^{p}(X,\te), d_{p})$, where the metric geodesics are described by the weak geodesics whose construction does not depend on $p$. 

\begin{lem} \label{lem: Geodesic chords are the same}
    If $1 \leq p' \leq p$, then $\cR^{p}_{\te} \subset \cR^{p'}_{\te}$. Moreover, if $[0,1] \ni \al \mapsto \{u_{t}^{\al}\}_{t} \in \cR^{p}_{\te}$ is the $d_{p}^{c}$-geodesic joining $\{u_{t}^{0}\}_{t}$ and $\{u_{t}^{1}\}_{t}$, then $[0,1]\ni \al\mapsto \{u_{t}^{\al}\}_{t}$ is also the $d_{p'}^{c}$-geodesic. 
\end{lem}

\begin{proof}
    Recall from the proof of Theorem~\ref{thm: geodesics chords for p> 1} that the the geodesics $[0,1] \ni \al \mapsto v_{t}^{\al} \in \E^{p}(X,\te)$ joining $u_{t}^{0}$ and $u_{t}^{1}$ do not depend on the value of $p$. Similarly, the geodesics $[0,t] \ni l \mapsto w_{l}^{\al,t} \in \E^{p}(X,\te)$ joining $V_{\te}$ and $v_{t}^{\al}$ also do not depend on the value of $p$. 

    If we can show that the $d_{p}$ limit $u_{l}^{\al}$ of $w_{l}^{\al,t}$ as $t \to \infty$ does not depend on $p$ as well, we will be done. Recall from Lemma~\ref{lem: one side inequality for dp distances}, that $d_{p'}(\cdot, \cdot) \leq d_{p}(\cdot, \cdot) \vol(\te)^{\frac{1}{p'}- \frac{1}{p}}$ on $\E^{p}(X,\te)$.  Thus if $\lim_{t\to\infty}d_{p}(w_{l}^{\al,t}, u_{l}^{\al}) = 0$, then 
    \[
    \lim_{t \to \infty}d_{p'}(w_{l}^{\al,t}, u_{l}^{\al}) \leq \lim_{t \to \infty}d_{p}(w_{l}^{\al,t}, u_{l}^{\al})\vol(\te)^{\frac{1}{p'}-\frac{1}{p}} = 0.
    \]
    Knowing that the $d_{p'}$ limit of $w_{l}^{\al,t}$ as $t \to \infty$ is $u_{l}^{\al}$ irrespective of the value of $p'$, the proof of the fact that $[0,1] \ni \al \mapsto \{u_{t}^{\al}\}_{t} \in \cR^{p'}_{\te}$ is a $d_{p'}^{c}$-geodesic goes through as in Theorem~\ref{thm: geodesics chords for p> 1}. 
\end{proof}

Now we will establish the version of Theorem~\ref{thm: geodesics chords for p> 1} for $p =1$. Before that, we need the following 

\begin{lem}\label{lem: monotone sequence of rays converge}
    For $p\geq 1$, if $\{u_{t}^{j}\}_{t},  \{u_{t}\} \in \cR^{p}_{\te}$ are finite energy geodesic rays such that $u_{t}^{j} \searrow u_{t}$ then 
    \[
    \lim_{j \to \infty} d_{p}^{c}(\{u_{t}^{j}\}_{t}, \{u_{t}\}_{t}) = 0.
    \]
\end{lem}

\begin{proof}
    By Lemma~\ref{lem: sup is affine in weak geodesics starting from V_{te}}, the map $t \mapsto \sup_{X}u_{t}$ and $t \mapsto \sup_{X}u_{t}^{j}$ is linear. Therefore, by replacing $u_{t}, u_{t}^{j}$ by $u_{t} - Ct, u_{t}^{j} - Ct$, we can assume that $0 \geq  u_{t}^{j} \geq u_{t}$. 

    Applying the Lidskii-type inequality from Theorem~\ref{thm: Lidskii inequality for big classes}, we get that
    \[
    \frac{d_{p}^{p}(u_{t}^{j}, u_{t})}{t^{p}} \leq \frac{d_{p}^{p}(u_{t}, 0)- d_{p}^{p}(u_{t}^{j}, 0)}{t^{p}} = d_{p}^{p}(u_{1}, 0) - d_{p}^{p}(u_{1}^{j}, 0).
    \]
    Since $u_{1}^{j} \searrow u_{1}$, by Lemma~\ref{lem: continuity under decreasing sequences}, $\lim_{j \to \infty} d_{p}(u_{1}^{j}, 0 ) = d_{p}(u_{1}, 0)$. Thus 
    \[
    \lim_{j \to \infty} \frac{d_{p}(u_{t}^{j}, u_{t})}{t} = 0
    \]
    uniformly in $t$. Therefore, $\lim_{j \to \infty}d_{p}^{c}(\{u_{t}^{j}\}_{t}, \{u_{t}\}_{t}) = 0$.
\end{proof}

With the help of Lemma~\ref{lem: monotone sequence of rays converge} and Theorem~\ref{thm: approximting with minimal singularity type}, we can prove the existence of geodesics in $\cR^{1}_{\te}$. 

\begin{thm}
    For $p \geq 1$, the space $(\cR^{p}_{\te}, d_{p}^{c})$ is a complete geodesic metric. 
\end{thm}

\begin{proof}
    We have already proved the theorem for $p > 1$. We just need to prove it for $p = 1$ now. Let $\{u_{t}^{0}\}_{t}, \{u_{t}^{1}\}_{t} \in \cR^{1}_{\te}$ be two finite energy geodeisc rays. As in the proof of Theorem~\ref{thm: geodesics chords for p> 1}, we can construct geodesics $[0,1] \ni \al \mapsto v_{t}^{\al} \in \E^{1}(X,\te)$ joining $u_{t}^{0}$ and $u_{t}^{1}$. We can also construct geodesics $[0,t] \ni l \mapsto w_{l}^{\al,t} \in \E^{1}(X,\te)$ joining $V_{\te}$ and $v_{t}^{\al}$. We used $p > 1$ to construct $u_{l}^{\al}$ as the $d_{p}$ limit of $w_{l}^{\al,t}$ as $t \to \infty$. In the case $p = 1$, we will use approximation to construct $u_{l}^{\al}$. 
    
    Fix any $p > 1$. From Theorem~\ref{thm: approximting with minimal singularity type} we can construct geodesic rays $\{u_{t}^{0,j}\}_{t}, \{u_{t}^{1,j}\}_{t} \in \cR^{p}_{\te}$, such that $u_{t}^{0,j} \searrow u_{t}^{0}$ and $u_{t}^{1,j} \searrow u_{t}^{1}$ as $j \to \infty$ and for all $t \geq 0$. 

    Let $[0,1] \ni \al \mapsto \{u_{t}^{\al,j}\}_{t} \in \cR^{p}_{\te} \subset \cR^{1}_{\te}$ be the $d_{p}^{c}$-geodesic constructed in Theorem~\ref{thm: geodesics chords for p> 1}. By Lemma~\ref{lem: Geodesic chords are the same} it is also the $d_{1}^{c}$-geodesic. Recall that we constructed $u_{l}^{\al, j}$ as $d_{p}$-limit (hence $d_{1}$-limit) of $w_{l}^{\al, t, j}$ as $t \to \infty$. Since $w_{l}^{\al, t , j} \geq w_{l}^{\al, t}$, we have 
    \begin{align*}
        d_{1}(w_{l}^{\al,t,j}, w_{l}^{\al, t}) &= I(w_{l}^{\al, t, j}) - I(w_{l}^{\al, t}) \\
        &= \frac{l}{t}( I(v_{t}^{\al, j}) - I(v_{t}^{\al})) \\
        &= \frac{l}{t} ((1-\al)I(u_{t}^{0,j}) + I(u_{t}^{1,j}) - (1-\al)I(u_{t}^{0}) - \al I(u_{t}^{1}))\\
        &= l(1-\al)(I(u_{1}^{0,j}) - I(u_{1}^{0})) + l\al (I(u_{1}^{1,j}) - I(u_{1}^{1})).
    \end{align*}
    As $j \to \infty$, $u_{1}^{0,j} \searrow u_{1}^{0}$ and $u_{1}^{1,j} \searrow u_{1}^{1}$. Therefore $I(u_{1}^{0,j}) \to I(u_{1}^{0})$ and $I(u_{1}^{1,j}) \to I(u_{1}^{1})$. Therefore, $d_{1}(w_{l}^{\al,t,j}, w_{l}^{\al, t}) \to 0$ as $j \to \infty$ uniformly over $t$. 

    Now since $u_{l}^{\al, j}$ is a decreasing sequence and is bounded from below because $I(u_{l}^{\al, j})$ is bounded. Therefore, the limit $\lim_{j \to \infty} u_{l}^{\al, j} =: u_{l}^{\al} \in \E^{1}(X,\te)$. By the proof of Theorem~\ref{thm: geodesics chords for p> 1}, we will be done, if we can show that $\lim_{t \to \infty}d_{1}(w_{l}^{\al,t}, u_{l}^{\al}) = 0 $. To see this notice that 
    \[
    d_{1}(w_{l}^{\al,t}, u_{l}^{\al}) \leq d_{1}(w_{l}^{\al, t}, w_{l}^{\al, t, j}) + d_{1}(w_{l}^{\al, t, j}, u_{l}^{\al, j}) + d_{1}(u_{l}^{\al, j}, u_{l}^{\al}).
    \]
    We have seen above that the first term goes to 0 as $j \to \infty$ uniformly in $t$. The last term also goes to 0 as $ j \to \infty$. For a fixed $j$, the second term gets arbitrarily small for large $t$. Therefore, $\lim_{t \to \infty} d_{1}(w_{l}^{\al,t}, u_{l}^{\al}) \to 0$. 
\end{proof}

\section{The Ross-Witt Nystr\"om correspondence}\label{sec: Ross-Witt Nystrom correspondence}

The notion of test curves, introduced by Ross-Witt Nystr\"om in \cite{RossWittNystromanalytictestconfigurations} in the K\"ahler setting, and further studied by Darvas-Di Nezza-Lu \cite{darvas2021l1}, Darvas-Zhang \cite{darvas2023twisted}, Darvas-Xia \cite{darvasxiaclosureoftestconfigurations} and Darvas-Xia-Zhang \cite{darvas2023nonarchimedeanmetrics}, is a powerful tool to study geodesic rays in connection with $K$-stability.

\begin{mydef}\label{def: test curve} A map $\R \ni \tau \mapsto \psi_{\tau} \in \PSH(X,\te)$ is a test curve, denoted by $\{\psi_{\tau}\}_{\tau}$, if 
\begin{enumerate}
    \item The map $\R \ni \tau \mapsto \psi_{\tau}(x)$ is decreasing, usc, concave for all $x \in X$, and
    \item $\psi_{\tau} = -\infty$ for all $\tau$ big enough, and
    \item $\psi_{\tau} \nearrow V_{\te}$ a.e. as $\tau \to -\infty$.
\end{enumerate}

\end{mydef}
From condition (2) above, we can define
\[
\tau_{\psi}^{+} = \inf\{ \tau \in \R \,:\, \psi_{\tau} \equiv - \infty\}.
\]

Test curves are Legendre dual of sublinear subgeodesic rays starting at $V_{\te}$ \cite{darvas2023twisted}. Several properties of the geodesic rays can be detected by the corresponding test curves. Based on that we define several subclasses of test curves.

\begin{mydef}\label{def: properties of test curve} A test curve $\{\psi_{\tau}\}_{\tau}$ can have the following properties.
\begin{enumerate}
    \item It is \emph{maximal} if $P_{\te}[\psi_{\tau}](0) = \psi_{\tau}$ for all $\tau \in \R$.
    \item It has \emph{finite $p$-energy} for $p \geq 1$, if 
    \[
     \int_{-\infty}^{\tau^{+}_{\psi}}(-\tau + \tau^{+}_{\psi})^{p-1}\left(\int_{X}\te^{n}_{V_{\te}} - \int_{X}\te^{n}_{\psi_{\tau}} \right) d\tau < \infty
    \]
    \item It is \emph{bounded} if for all $\tau$ negative enough, $\psi_{\tau} = V_{\te} $. In this case, we define 
    \[
    \tau^{-}_{\psi} = \sup\{ \tau \in \R \,:\, \psi_{\tau} = V_{\te}\}.
    \]
\end{enumerate}
    
\end{mydef}

When $p = 1$, the notion of finite $p$-energy test curves was introduced by Darvas-Xia \cite{darvasxiaclosureoftestconfigurations} in the K\"ahler setting, and by Darvas-Zhang \cite{darvas2023twisted} in the big setting. 

The Ross-Witt Nystr\"om correspondence is the observation of Ross-Witt Nystr\"om \cite{RossWittNystromanalytictestconfigurations} that test curves and subgeodesic rays are dual to each other via Legendre transform. If $\{u_{t}\}_{t}$ is a sublinear sub geodesic ray starting at $V_{\te}$, then its Legendre transform is given by
\[
\hat{u}_{\tau}(x) = \inf_{t > 0} (u_{t}(x) - t\tau). 
\]

If $\{\psi_{\tau}\}_{\tau}$ is a test curve, then its inverse Legendre transform is given by 
\[
\check{\psi}_{t}(x) = \sup_{\tau \in \R} (\psi_{\tau}(x) - t\tau).
\]

In particular, we have 
\begin{thm}[{\cite[Theorem 3.7]{darvas2023twisted}}] The Legendre transform $\{\psi_{\tau}\}_{\tau} \mapsto \{\check{\psi}_{t}\}_{t}$ is a bijection with inverse $\{u_{t}\}_{t} \mapsto \{\hat{u}_{\tau}\}_{\tau}$ between
    \begin{enumerate}
        \item test curves and sublinear subgeodesic rays starting at $V_{\te}$; 
        \item maximal test curves and geodesic rays starting at $V_{\te}$; 
        \item bounded maximal test curves and geodesic rays with minimal singularity type. Moreover, in this case, $V_{\te} + \tau^{-}_{\psi}t \leq \check{\psi}_{t} \leq V_{\te} + \tau^{+}_{\psi}t $.
    \end{enumerate}
\end{thm}

Darvas-Xia \cite[Theorem 3.7]{darvasxiaclosureoftestconfigurations} in the K\"ahler setting, and Darvas-Zhang \cite[Theorem 3.9]{darvas2023twisted} in the big setting proved that the Ross-Witt Nystr\"om correspondence as described above is a bijection between maximal finite $1$-energy test curves and finite $1$-energy geodesic rays. In this section, we will generalize this statement to finite $p$-energy test curves and geodesic rays. 

For that, first, we need a lemma about the speed of the geodesic segments, which can be seen as a converse of \cite[Lemma 5.3]{guptacompletegeodescimetricinbigclasses}.

\begin{lem}\label{lem: Finite geodesic speed implies finite energy potential}
    If $u_{0} \in \cH_{\te} $, $u_{1} \in \PSH(X,\te)$, and $[0,1] \ni t \mapsto u_{t} \in \PSH(X,\te)$ is the weak geodesic joining $u_{0}$ and $u_{1}$ such that 
    \[
        \int_{X} |\dot{u}_{0}|^{p}\te^{n}_{u_{0}} < \infty,
    \]
    then $u_{1} \in \E^{p}(X,\te)$.
\end{lem}

\begin{proof}
    First, assume that $u_{0} \geq u_{1} + 1$, so that we can find $u_{1}^{j} \in \cH_{\te}$ such that $u_{1}^{j} \searrow u_{1}$ and $u_{0} \geq u_{1}^{j}$. Let $[0,1] \ni t \mapsto u_{t}^{j} \in \E^{p}(X,\te)$ be the weak geodesic joining $u_{0}$ and $u_{1}^{j}$. The geodesics are decreasing, therefore $\dot{u}_{0}, \dot{u}_{0}^{j} \leq 0$.  By Theorem~\ref{thm: geodesic speed and distance}, 
    \[
    d_{p}^{p}(u_{0}, u_{1}^{j}) = \int_{X} |\dot{u}_{0}^{j}|^{p}\te^{n}_{u_{0}} = \int_{X} (-\dot{u}_{0}^{j})^{p}\te^{n}_{u_{0}}.
    \]
    Since $u_{1}^{j} \geq u_{1}$, the weak geodesics satisfy $u_{t}^{j} \geq u_{t}$. Therefore,
    \[
    \dot{u}_{0}^{j} = \lim_{t \to 0} \frac{u_{t}^{j} - u_{0}}{t} \geq \lim_{t \to 0} \frac{u_{t} - u_{0}}{t} = \dot{u}_{0}.
    \]
    Hence, $(-\dot{u}_{0}^{j})^{p} \leq (-\dot{u}_{0})^{p}$. Now, 
    \[
    d_{p}^{p}(u_{0}, u_{1}^{j}) =  \int_{X}(-\dot{u}_{0}^{j})^{p}\te^{n}_{u_{0}} \leq \int_{X}|\dot{u}_{0}|^{p}\te^{n}_{u_{0}} < \infty.
    \]
    Therefore, $u_{1}^{j}$ is a $d_{p}$-bounded decreasing sequence such that $u_{1}^{j} \searrow u_{1}$. Therefore by Lemma~\ref{lem: decreasing bounded dp sequence has a limit in Ep} $u_{1} \in \E^{p}(X,\te)$.

    Now we drop the assumption that $u_{0} \geq u_{1} + 1$. Let $C > 0$ such that $u_{1} - C + 1 \leq u_{0}$. Let $w_{1} = u_{1} - C $, so that $u_{0} \geq w_{1} + 1$. If $w_{t}$ is the geodesic joining $u_{0}$ and $w_{1}$, then $w_{t} = u_{t} - Ct$. Therefore, $\dot{w}_{0} = \dot{u}_{0} - C$. Hence, 
    \[
    \int_{X}|\dot{w}_{0}|^{p}\te^{n}_{u_{0}} = \int_{X}|\dot{u}_{0} - C|^{p}\te^{n}_{u_{0}} \leq 2^{p-1}\left( \int_{X}|\dot{u}_{0}|^{p}\te^{n}_{u_{0}} + C^{p}\vol(\te)\right) < \infty.
    \]
    By the argument above, we find that $w_{1} \in \E^{p}(X,\te)$, therefore, $w_{1} + C = u_{1} \in \E^{p}(X,\te)$ as well. 
\end{proof}

This Lemma allows us to prove

\begin{thm}\label{thm: The Ross-Witt Nystrom correspondence}
    For $p \geq 1$, the Legendre Transform $\{u_{t}\}_{t} \mapsto \{\hat{u}_{\tau}\}_{\tau}$ is the bijective map between finite $p$-energy geodesic rays in $\cR^{p}_{\te}$ and maximal finite $p$-energy test curves. 
\end{thm}

\begin{proof}
    First, we assume that $\sup_{X}u_{t} = 0$ for all $t \geq 0$. In this case $\tau_{\hat{u}}^{+} = 0$. To see this we observe that $u_{t}$ decreases as $t \to \infty$. As $\sup_{X}u_{t} = 0$, by Hartog's Lemma (see \cite[Proposition 8.4]{guedj2017degenerate}) $\sup_{X}(\lim_{t\to \infty} u_{t}) = 0$. Therefore, when $\tau = 0$, 
    \[
    \hat{u}_{0} = \inf_{t \geq 0}u_{t} = \lim_{t \to \infty} u_{t} \not\equiv -\infty
    \]
    Whereas for $\tau > 0$, 
    \[
    \hat{u}_{\tau} = \inf_{t \geq 0}(u_{t} - t\tau) \equiv -\infty. 
    \]
    Thus $\tau^{+}_{\hat{u}} = 0$.  Again since $u_{t}$ is decreasing $\dot{u}_{0} \leq 0$. We will show that 
    \begin{equation}\label{eq: geodesic speed at 0 and energy of the test curve}
    \int_{X} |\dot{u}_{0}|^{p} \te^{n}_{u_{0}} = p\int_{-\infty}^{0} (-\tau)^{p-1}\left(\int_{X} \te^{n}_{V_{\te}} - \int_{X} \te^{n}_{\hat{u}_{\tau}}\right) d\tau.
    \end{equation}
    To see this we start with the left-hand side
    \begin{align*}
        \int_{X}|\dot{u}_{0}|^{p}\te^{n}_{u_{0}} &= p\int_{X}(-\dot{u}_{0})^{p}\te^{n}_{V_{\te}} \\
        &= p\int_{0}^{\infty} \tau^{p-1} \te^{n}_{V_{\te}} (\{\dot{u}_{0} < -\tau\}) d\tau. 
        \intertext{Changing the variable from $\tau$ to $-\tau$, we can continue}
        &= p\int_{-\infty}^{0}(-\tau)^{p-1}\te^{n}_{V_{\te}}(\{ \dot{u}_{0} < \tau\})d\tau.
        \intertext{Recall that $\hat{u}_{\tau} = \inf_{t \geq 0} (u_{t} - t\tau)$. Thus the set $\{\dot{u}_{0} \geq \tau\} = \{ \hat{u}_{\tau} = u_{0} = V_{\te}\}.$ Therefore,}
        &=p \int_{-\infty}^{0}(-\tau)^{p-1}\left( \vol(\te) - \te^{n}_{V_{\te}}(\{\hat{u}_{\tau} = V_{\te}\}) \right) d\tau. 
        \intertext{Since $\{\hat{u}_{\tau}\}$ is a maximal test curve, we know that $P_{\te}[\hat{u}_{\tau}] = \hat{u}_{\tau}$. Therefore, by Lemma~\ref{lem: measure on contact set}, $\te^{n}_{P_{\te}[\hat{u}_{\tau}]} = \mathds{1}_{\{P_{\te}[\hat{u}_{\tau}] = 0\}}\te^{n} $. Moreover, $\te^{n}_{V_{\te}} = \mathds{1}_{\{ V_{\te} = 0\}} \te^{n}$. As $P_{\te}[\hat{u}_{\tau}] = \hat{u}_{\tau}$, and $\{ \hat{u}_{\tau} =  0 \} \subset \{V_{\te} = 0\}$, we get that $\te^{n}_{\hat{u}_{\tau}} = \mathds{1}_{\{ \hat{u}_{\tau} = 0\}} \te^{n}_{V_{\te}}. $ Thus we can write, $\te^{n}_{V_{\te}} (\{\hat{u}_{\tau} = V_{\te}\}) = \int_{X}\te^{n}_{\hat{u}_{\tau}}$. So we get}
        &=p \int_{-\infty}^{0}(-\tau)^{p-1}\left(\int_{X} \te^{n}_{V_{\te}} - \int_{X} \te^{n}_{\hat{u}_{\tau}} \right) d\tau.
    \end{align*}

This proves \eqref{eq: geodesic speed at 0 and energy of the test curve}. 

If the test curve $\{\hat{u}_{\tau}\}_{\tau}$ has finite $p$-energy, then by \eqref{eq: geodesic speed at 0 and energy of the test curve}, $\int_{X}|\dot{u}_{0}|^{p}\te^{n}_{u_{0}} < \infty$. Since $[0,1] \ni s \mapsto w_{s}:=  u_{ts}$ is the geodesic joining $u_{0}$ and $u_{t}$. Thus, $\dot{w}_{0} = t\dot{u}_{0}$ and
\[
\int_{X} |\dot{w}_{0}|^{p}\te^{n}_{u_{0}} = t^{p}\int_{X} |\dot{u}_{0}|^{p}\te^{n}_{u_{0}} < \infty
\]
Therefore, by Lemma~\ref{lem: Finite geodesic speed implies finite energy potential},  $u_{t} \in \E^{p}(X,\te)$. Thus the geodesic ray $\{u_{t}\}_{t}$ is a finite $p$-energy geodesic ray. 

Similarly, if $\{u_{t}\}_{t}$ is a finite $p$-energy geodesic ray then $\int_{X} |\dot{u}_{0}|^{p}\te^{n}_{u_{0}} = d_{p}^{p}(u_{0}, u_{1}) < \infty$, thus by \eqref{eq: geodesic speed at 0 and energy of the test curve}, $\{\hat{u}_{\tau}\}_{\tau}$ is a finite $p$-energy test curve. 

More generally, by Lemma~\ref{lem: sup is affine in weak geodesics starting from V_{te}}, $\sup_{X}u_{t} = Ct$ for some $C$. Moreover, by the same argument as earlier $C = \tau_{\hat{u}}^{+}$.  If $w_{t} = u_{t} - Ct = u_{t} - \tau_{\hat{u}}^{+}t$, then $\{w_{t}\}_{t}$ is a geodesic whose Legendre dual $\{\hat{w}_{\tau}\}_{\tau}$ is given by $\hat{w}_{\tau} = \hat{u}_{\tau +C} = \hat{u}_{\tau + \tau_{\hat{u}}^{+}}$.

From \eqref{eq: geodesic speed at 0 and energy of the test curve}, we can say that 
\begin{align*}
\int_{X}|\dot{w}_{0}|^{p}\te^{n}_{V_{\te}} &= p\int_{-\infty}^{0} (-\tau)^{p-1}\left( \int_{X}\te^{n}_{V_{\te}} - \int_{X} \te^{n}_{\hat{w}_{\tau}}\right)d\tau \\
&=p \int_{-\infty}^{0}(-\tau)^{p-1}\left( \int_{X}\te^{n}_{V_{\te}} - \int_{X} \te^{n}_{\hat{u}_{\tau + \tau_{\hat{u}}^{+}}}\right)d\tau \\
&= p\int_{-\infty}^{\tau_{\hat{u}}^{+}} (-\tau + \tau_{\hat{u}}^{+})^{p-1} \left( \int_{X} \te^{n}_{V_{\te}} - \int_{X} \te^{n}_{\hat{u}_{\tau}}\right) d\tau
\end{align*}

As $\dot{w}_{0} = \dot{u}_{0} - \tau_{\hat{u}}^{+}$, we get that 
\[
\int_{X}|\dot{u}_{0}|^{p}\te^{n}_{V_{\te}} < \infty \iff \int_{X} |\dot{w}_{0}|^{p}\te^{n}_{V_{\te}} < \infty.
\]
Combining the two equations above, we get that $\{u_{t}\}_{t} \in \cR^{p}_{\te}$ is the finite $p$-energy geodesic ray iff the Legendre transform $\{\hat{u}_{\tau}\}_{\tau}$ is a maximal finite $p$-energy test curve. 
\end{proof}

\printbibliography

@article{Darvas2019GeometricPT,
    AUTHOR = {Darvas, Tam\'{a}s},
     TITLE = {Geometric pluripotential theory on {K}\"{a}hler manifolds},
 BOOKTITLE = {Advances in complex geometry},
    SERIES = {Contemp. Math.},
    VOLUME = {735},
     PAGES = {1--104},
 PUBLISHER = {Amer. Math. Soc., [Providence], RI},
      YEAR = {2019},
   MRCLASS = {32Q15 (32Q20 53C25 53C55)},
  MRNUMBER = {3996485},
       DOI = {10.1090/conm/735/14822},
}

@article {Guedj2019PlurisubharmonicEA,
    AUTHOR = {Guedj, Vincent and Lu, Chinh H. and Zeriahi, Ahmed},
     TITLE = {Plurisubharmonic envelopes and supersolutions},
   JOURNAL = {J. Differential Geom.},
  FJOURNAL = {Journal of Differential Geometry},
    VOLUME = {113},
      YEAR = {2019},
    NUMBER = {2},
     PAGES = {273--313},
      ISSN = {0022-040X},
   MRCLASS = {32W20 (32Q15 32U05)},
  MRNUMBER = {4023293},
MRREVIEWER = {Eleonora Di Nezza},
       DOI = {10.4310/jdg/1571882428},
}

@book {guedj2017degenerate,
    shorthand = {GZBook},
    AUTHOR = {Guedj, Vincent and Zeriahi, Ahmed},
     TITLE = {Degenerate complex {M}onge-{A}mp\`ere equations},
    SERIES = {EMS Tracts in Mathematics},
    VOLUME = {26},
 PUBLISHER = {European Mathematical Society (EMS), Z\"{u}rich},
      YEAR = {2017},
     PAGES = {xxiv+472},
      ISBN = {978-3-03719-167-5},
   MRCLASS = {32W20 (32Q20 32U15 32U20 32U40 35J96)},
  MRNUMBER = {3617346},
MRREVIEWER = {Slimane Benelkourchi},
       DOI = {10.4171/167},
}

@article {Darvas2017OnTS,
    AUTHOR = {Darvas, Tam\'{a}s and Di Nezza, Eleonora and Lu, Chinh H.},
     TITLE = {On the singularity type of full mass currents in big
              cohomology classes},
   JOURNAL = {Compos. Math.},
  FJOURNAL = {Compositio Mathematica},
    VOLUME = {154},
      YEAR = {2018},
    NUMBER = {2},
     PAGES = {380--409},
      ISSN = {0010-437X},
   MRCLASS = {32W20 (32Q15 32U05 53C55)},
  MRNUMBER = {3738831},
MRREVIEWER = {Vincent Guedj},
       DOI = {10.1112/S0010437X1700759X},
}

@article {Boucksom2008MongeAmpreEI,
shorthand = {BEGZ10},
    AUTHOR = {Boucksom, S\'{e}bastien and Eyssidieux, Philippe           and Guedj, Vincent and Zeriahi, Ahmed},
     TITLE = {Monge-{A}mp\`ere equations in big cohomology classes},
   JOURNAL = {Acta Math.},
  FJOURNAL = {Acta Mathematica},
    VOLUME = {205},
      YEAR = {2010},
    NUMBER = {2},
     PAGES = {199--262},
      ISSN = {0001-5962},
   MRCLASS = {32U40 (32Q20 32U15 32W20)},
  MRNUMBER = {2746347},
MRREVIEWER = {S\l awomir Dinew},
       DOI = {10.1007/s11511-010-0054-7},
}

@article{HespaceofKahlerpotentials,
author = {He, Weiyong},
title = {On the Space of Kähler Potentials},
journal = {Communications on Pure and Applied Mathematics},
volume = {68},
number = {2},
pages = {332-343},
doi = {https://doi.org/10.1002/cpa.21515},
year = {2015}
}

@article {BDL17,
    AUTHOR = {Berman, Robert J. and Darvas, Tam\'{a}s and Lu, Chinh H.},
     TITLE = {Convexity of the extended {K}-energy and the large time
              behavior of the weak {C}alabi flow},
   JOURNAL = {Geom. Topol.},
  FJOURNAL = {Geometry \& Topology},
    VOLUME = {21},
      YEAR = {2017},
    NUMBER = {5},
     PAGES = {2945--2988},
      ISSN = {1465-3060},
   MRCLASS = {53C55 (32U05 32W20)},
  MRNUMBER = {3687111},
MRREVIEWER = {Chi Li},
       DOI = {10.2140/gt.2017.21.2945},
}

@article {darvas2021l1,
    AUTHOR = {Darvas, Tam\'{a}s and Di Nezza, Eleonora and Lu, Chinh H.},
     TITLE = {{$L^1$} metric geometry of big cohomology classes},
   JOURNAL = {Ann. Inst. Fourier (Grenoble)},
  FJOURNAL = {Universit\'{e} de Grenoble. Annales de l'Institut Fourier},
    VOLUME = {68},
      YEAR = {2018},
    NUMBER = {7},
     PAGES = {3053--3086},
      ISSN = {0373-0956,1777-5310},
   MRCLASS = {32Q15 (32U40 32W20 53C55)},
  MRNUMBER = {3959105},
MRREVIEWER = {Valentino\ Tosatti},
       URL = {http://aif.cedram.org/item?id=AIF_2018__68_7_3053_0},
}

@article {dinezza2018lp,
    AUTHOR = {Di Nezza, Eleonora and Lu, Chinh H.},
     TITLE = {{$L^p$} metric geometry of big and nef cohomology classes},
   JOURNAL = {Acta Math. Vietnam.},
  FJOURNAL = {Acta Mathematica Vietnamica},
    VOLUME = {45},
      YEAR = {2020},
    NUMBER = {1},
     PAGES = {53--69},
      ISSN = {0251-4184},
   MRCLASS = {53C55 (32Q15 32W20 53C25)},
  MRNUMBER = {4081364},
MRREVIEWER = {Daniel Greb},
       DOI = {10.1007/s40306-019-00343-4},
}

@article{TrusianiL1metric,
  title={ $L^{1}$  Metric Geometry of Potentials with Prescribed Singularities on Compact K{\"a}hler Manifolds},
  author={Antonio Trusiani},
  journal={The Journal of Geometric Analysis},
  year={2022},
  volume={32},
  pages={1-37}
}

@article {phongsturmtestconfigurations,
    AUTHOR = {Phong, Duong H. and Sturm, Jacob},
     TITLE = {Test configurations for {K}-stability and geodesic rays},
   JOURNAL = {J. Symplectic Geom.},
  FJOURNAL = {The Journal of Symplectic Geometry},
    VOLUME = {5},
      YEAR = {2007},
    NUMBER = {2},
     PAGES = {221--247},
      ISSN = {1527-5256,1540-2347},
   MRCLASS = {32Q15 (32L10 32W20 58E11)},
  MRNUMBER = {2377252},
MRREVIEWER = {Julien\ Keller},
       URL = {http://projecteuclid.org/euclid.jsg/1202004456},
}

@article {mabuchiriemannianstructure,
    AUTHOR = {Mabuchi, Toshiki},
     TITLE = {Some symplectic geometry on compact {K}\"{a}hler manifolds. {I}},
   JOURNAL = {Osaka J. Math.},
  FJOURNAL = {Osaka Journal of Mathematics},
    VOLUME = {24},
      YEAR = {1987},
    NUMBER = {2},
     PAGES = {227--252},
      ISSN = {0030-6126},
   MRCLASS = {53C55 (32C10 58B20 58D17 58E20)},
  MRNUMBER = {909015},
MRREVIEWER = {Yoshiko Kubo},
       URL = {http://projecteuclid.org.proxy-um.researchport.umd.edu/euclid.ojm/1200780161},
}

@article {semmesriemannianmetric,
    AUTHOR = {Semmes, Stephen},
     TITLE = {Complex {M}onge-{A}mp\`ere and symplectic manifolds},
   JOURNAL = {Amer. J. Math.},
  FJOURNAL = {American Journal of Mathematics},
    VOLUME = {114},
      YEAR = {1992},
    NUMBER = {3},
     PAGES = {495--550},
      ISSN = {0002-9327},
   MRCLASS = {32F07 (35J60 58F05)},
  MRNUMBER = {1165352},
       DOI = {10.2307/2374768},
}

@incollection {donaldsonriemannianmetric,
    AUTHOR = {Donaldson, S. K.},
     TITLE = {Symmetric spaces, {K}\"{a}hler geometry and {H}amiltonian
              dynamics},
 BOOKTITLE = {Northern {C}alifornia {S}ymplectic {G}eometry {S}eminar},
    SERIES = {Amer. Math. Soc. Transl. Ser. 2},
    VOLUME = {196},
     PAGES = {13--33},
 PUBLISHER = {Amer. Math. Soc., Providence, RI},
      YEAR = {1999},
   MRCLASS = {58B25 (32Q20 32W20 53C55 53D20 58E11)},
  MRNUMBER = {1736211},
MRREVIEWER = {Matthew B. Stenzel},
       DOI = {10.1090/trans2/196/02},
}

@article {chenspaceofkahlermetrics,
    AUTHOR = {Chen, Xiuxiong},
     TITLE = {The space of {K}\"{a}hler metrics},
   JOURNAL = {J. Differential Geom.},
  FJOURNAL = {Journal of Differential Geometry},
    VOLUME = {56},
      YEAR = {2000},
    NUMBER = {2},
     PAGES = {189--234},
      ISSN = {0022-040X},
   MRCLASS = {32Q15 (32W20 58D27 58E11)},
  MRNUMBER = {1863016},
MRREVIEWER = {David L. Finn},
       URL = {http://projecteuclid.org.proxy-um.researchport.umd.edu/euclid.jdg/1090347643},
}

@article {darvasgeoemtryoffiniteenergy,
    AUTHOR = {Darvas, Tam\'{a}s},
     TITLE = {The {M}abuchi geometry of finite energy classes},
   JOURNAL = {Adv. Math.},
  FJOURNAL = {Advances in Mathematics},
    VOLUME = {285},
      YEAR = {2015},
     PAGES = {182--219},
      ISSN = {0001-8708},
   MRCLASS = {53C55 (32U05 32W20 53C60)},
  MRNUMBER = {3406499},
MRREVIEWER = {Kai Zheng},
       DOI = {10.1016/j.aim.2015.08.005},
}

@article {darvascompletionofspaceofkahlerpotentials,
    AUTHOR = {Darvas, Tam\'{a}s},
     TITLE = {The {M}abuchi completion of the space of {K}\"{a}hler potentials},
   JOURNAL = {Amer. J. Math.},
  FJOURNAL = {American Journal of Mathematics},
    VOLUME = {139},
      YEAR = {2017},
    NUMBER = {5},
     PAGES = {1275--1313},
      ISSN = {0002-9327},
   MRCLASS = {32Q15 (32U05 32W20 53C55)},
  MRNUMBER = {3702499},
MRREVIEWER = {A. Yu. Rashkovski\u{\i}},
       DOI = {10.1353/ajm.2017.0032},
}

@article {darvasrubinsteinpropernessconjecture,
    AUTHOR = {Darvas, Tam\'{a}s and Rubinstein, Yanir A.},
     TITLE = {Tian's properness conjectures and {F}insler geometry of the
              space of {K}\"{a}hler metrics},
   JOURNAL = {J. Amer. Math. Soc.},
  FJOURNAL = {Journal of the American Mathematical Society},
    VOLUME = {30},
      YEAR = {2017},
    NUMBER = {2},
     PAGES = {347--387},
      ISSN = {0894-0347},
   MRCLASS = {32Q20 (14J45 32U05 32W20 58B20 58E11)},
  MRNUMBER = {3600039},
MRREVIEWER = {Eleonora Di Nezza},
       DOI = {10.1090/jams/873},
}

@article {DarvasLuRubinsteinQuantization,
    AUTHOR = {Darvas, Tam\'{a}s and Lu, Chinh H. and Rubinstein, Yanir A.},
     TITLE = {Quantization in geometric pluripotential theory},
   JOURNAL = {Comm. Pure Appl. Math.},
  FJOURNAL = {Communications on Pure and Applied Mathematics},
    VOLUME = {73},
      YEAR = {2020},
    NUMBER = {5},
     PAGES = {1100--1138},
      ISSN = {0010-3640,1097-0312},
   MRCLASS = {32U99 (32Q15 53C55 53C60 53D50)},
  MRNUMBER = {4078714},
MRREVIEWER = {Marcin\ Sroka},
       DOI = {10.1002/cpa.21857},
}

@article {BermanMAequationzerotemplimit,
    AUTHOR = {Berman, Robert J.},
     TITLE = {From {M}onge-{A}mp\`ere equations to envelopes and geodesic
              rays in the zero temperature limit},
   JOURNAL = {Math. Z.},
  FJOURNAL = {Mathematische Zeitschrift},
    VOLUME = {291},
      YEAR = {2019},
    NUMBER = {1-2},
     PAGES = {365--394},
      ISSN = {0025-5874,1432-1823},
   MRCLASS = {32W20 (32Q15)},
  MRNUMBER = {3936074},
MRREVIEWER = {Vincent\ Guedj},
       DOI = {10.1007/s00209-018-2087-0}
}

@article {chenchengcsckI,
    AUTHOR = {Chen, Xiuxiong and Cheng, Jingrui},
     TITLE = {On the constant scalar curvature {K}\"{a}hler metrics ({I})---{A}
              priori estimates},
   JOURNAL = {J. Amer. Math. Soc.},
  FJOURNAL = {Journal of the American Mathematical Society},
    VOLUME = {34},
      YEAR = {2021},
    NUMBER = {4},
     PAGES = {909--936},
      ISSN = {0894-0347},
   MRCLASS = {53C55 (53C21)},
  MRNUMBER = {4301557},
       DOI = {10.1090/jams/967},
}

@article {chenchengcscKII,
    AUTHOR = {Chen, Xiuxiong and Cheng, Jingrui},
     TITLE = {On the constant scalar curvature {K}\"{a}hler metrics
              ({II})---{E}xistence results},
   JOURNAL = {J. Amer. Math. Soc.},
  FJOURNAL = {Journal of the American Mathematical Society},
    VOLUME = {34},
      YEAR = {2021},
    NUMBER = {4},
     PAGES = {937--1009},
      ISSN = {0894-0347},
   MRCLASS = {53C55 (53C21)},
  MRNUMBER = {4301558},
       DOI = {10.1090/jams/966},
}

@misc{chenandchengcscKIII,
      title={On the constant scalar curvature K\"ahler metrics, general automorphism group}, 
      author={Xiuxiong Chen and Jingrui Cheng},
      year={2018},
      eprint={1801.05907},
      archivePrefix={arXiv},
      primaryClass={math.DG}
}

@article{darvasxiaclosureoftestconfigurations,
title = {The closures of test configurations and algebraic singularity types},
journal = {Advances in Mathematics},
volume = {397},
pages = {108198},
year = {2022},
issn = {0001-8708},
doi = {https://doi.org/10.1016/j.aim.2022.108198},
author = {Tamás Darvas and Mingchen Xia},
}

@article {GuptaCompletemetrictopology,
    AUTHOR = {Gupta, Prakhar},
     TITLE = {A complete metric topology on relative low energy spaces},
   JOURNAL = {Math. Z.},
  FJOURNAL = {Mathematische Zeitschrift},
    VOLUME = {303},
      YEAR = {2023},
    NUMBER = {3},
     PAGES = {Paper No. 56, 27},
      ISSN = {0025-5874},
   MRCLASS = {32U05 (32Q15 53E30 54E35)},
  MRNUMBER = {4546855},
       DOI = {10.1007/s00209-023-03218-5},
}

@misc{guptacompletegeodescimetricinbigclasses,
      title={Complete Geodesic Metrics in Big Classes}, 
      author={Prakhar Gupta},
      year={2024},
      eprint={2401.01688},
      archivePrefix={arXiv},
      primaryClass={math.DG}
}

@article {EleonoraTrapaniMAmeasureoncontactsets,
    AUTHOR = {Di Nezza, Eleonora and Trapani, Stefano},
     TITLE = {Monge-{A}mp\`ere measures on contact sets},
   JOURNAL = {Math. Res. Lett.},
  FJOURNAL = {Mathematical Research Letters},
    VOLUME = {28},
      YEAR = {2021},
    NUMBER = {5},
     PAGES = {1337--1352},
      ISSN = {1073-2780,1945-001X},
   MRCLASS = {32W20 (32Q15)},
  MRNUMBER = {4471711},
}

@article {Wittnystrommonotonicity,
    AUTHOR = {Witt Nystr\"{o}m, David},
     TITLE = {Monotonicity of non-pluripolar {M}onge-{A}mp\`ere masses},
   JOURNAL = {Indiana Univ. Math. J.},
  FJOURNAL = {Indiana University Mathematics Journal},
    VOLUME = {68},
      YEAR = {2019},
    NUMBER = {2},
     PAGES = {579--591},
      ISSN = {0022-2518,1943-5258},
   MRCLASS = {32W20 (32Q15 32U05 35J96)},
  MRNUMBER = {3951074},
MRREVIEWER = {Rafa\l \ Czy\.{z}},
       DOI = {10.1512/iumj.2019.68.7630},
}

@article {RossWittNystromanalytictestconfigurations,
    AUTHOR = {Ross, Julius and Witt Nystr\"{o}m, David},
     TITLE = {Analytic test configurations and geodesic rays},
   JOURNAL = {J. Symplectic Geom.},
  FJOURNAL = {The Journal of Symplectic Geometry},
    VOLUME = {12},
      YEAR = {2014},
    NUMBER = {1},
     PAGES = {125--169},
      ISSN = {1527-5256,1540-2347},
   MRCLASS = {32L05 (32U05)},
  MRNUMBER = {3194078},
MRREVIEWER = {Bianca\ Santoro},
       DOI = {10.4310/JSG.2014.v12.n1.a5},
}

@article {Berndtssonweakgeodesics,
    AUTHOR = {Berndtsson, Bo},
     TITLE = {A {B}runn-{M}inkowski type inequality for {F}ano manifolds and
              some uniqueness theorems in {K}\"{a}hler geometry},
   JOURNAL = {Invent. Math.},
  FJOURNAL = {Inventiones Mathematicae},
    VOLUME = {200},
      YEAR = {2015},
    NUMBER = {1},
     PAGES = {149--200},
      ISSN = {0020-9910,1432-1297},
   MRCLASS = {53C55 (32Q20 53C25)},
  MRNUMBER = {3323577},
MRREVIEWER = {Vincent\ Guedj},
       DOI = {10.1007/s00222-014-0532-1},
}

@misc{darvas2023nonarchimedeanmetrics,
      title={A transcendental approach to non-Archimedean metrics of pseudoeffective classes}, 
      author={Tamás Darvas and Mingchen Xia and Kewei Zhang},
      year={2023},
      eprint={2302.02541},
      archivePrefix={arXiv},
      primaryClass={math.AG}
}

@article{darvas2023relative,
      title={Relative pluripotential theory on compact K\"ahler manifolds}, 
      author={ Darvas, Tamás and  Di Nezza, Eleonora and  Lu, Chinh H.},
      year={2023},
      eprint={2303.11584},
      archivePrefix={arXiv},
      primaryClass={math.CV}
}

@misc{guedj2014metriccompletion,
      title={The metric completion of the Riemannian space of K\"{a}hler metrics}, 
      author={Vincent Guedj},
      year={2014},
      eprint={1401.7857},
      archivePrefix={arXiv},
      primaryClass={math.DG}
}

@article{darvas2023twisted,
      title={Twisted K\"ahler-Einstein metrics in big classes}, 
      author={Tamás Darvas and Kewei Zhang},
      year={2023},
      eprint={2208.08324},
      archivePrefix={arXiv},
      primaryClass={math.DG}
}

@article {DarvasLuuniformconvexity,
    AUTHOR = {Darvas, Tam\'{a}s and Lu, Chinh H.},
     TITLE = {Geodesic stability, the space of rays and uniform convexity in
              {M}abuchi geometry},
   JOURNAL = {Geom. Topol.},
  FJOURNAL = {Geometry \& Topology},
    VOLUME = {24},
      YEAR = {2020},
    NUMBER = {4},
     PAGES = {1907--1967},
      ISSN = {1465-3060,1364-0380},
   MRCLASS = {32Q26 (32U05 53C55)},
  MRNUMBER = {4173924},
MRREVIEWER = {Kai\ Zheng},
       DOI = {10.2140/gt.2020.24.1907},
}

@book{minchennotes,
    author = {Xia, Mingchen},
    title = {Singularities in global pluripotential theory},
    URL = {https://mingchenxia.github.io/home/Lectures/SGPT.pdf},
    year = {2024}
}

\end{document}